\documentclass[a4paper,10pt]{amsart}
\usepackage{amsmath,amsthm,latexsym,verbatim}
\usepackage[all]{xy}
\usepackage{pxfonts}
\usepackage[latin1]{inputenc}
\usepackage{amssymb}
\usepackage{fancyhdr}
\usepackage{mathrsfs}
\usepackage{tikz}
\usetikzlibrary{arrows}

\theoremstyle{definition}

\theoremstyle{plain}
\newtheorem{Thm}{Theorem}[section]
\newtheorem{Def}[Thm]{Definition}
\newtheorem{Lem}[Thm]{Lemma}
\newtheorem{ex}[Thm]{Example}

\newtheorem{Prop}[Thm]{Proposition}

\newtheorem{Propr}[Thm]{Property}
\newtheorem{rem}[Thm]{Remark}

\newtheorem{Ass}{Assumptions}

\newcommand{\supp}{\operatorname{supp}}

\newcommand{\singsupp}{\operatorname{singsupp}}

\newcommand{\ind}{\textnormal{ind-lim}}

\newcommand{\W}{\mathscr{W}}
\newcommand{\proj}{\textnormal{proj-lim}}

\newcommand{\R}{\mathbb{R}}

\newcommand{\C}{\mathbb{C}}

\newcommand{\Z}{\mathbb{Z}}
\newcommand{\N}{\mathbb{N}}
\newcommand{\Si}{\mathscr{S}}

\newcommand{\s}{\mathbf{s}}

\newcommand{\id}{\textnormal{Id}}

\newcommand{\tr}{\textrm{tr}}
\newcommand{\op}{\textnormal{Op}}

\usepackage{enumerate}

\newcommand{\dbar}{\;{\mathchar'26\mkern-10mu  d}}
\newcommand{\limproj}{\textnormal{proj-lim}}

\newcommand{\pt}[1]{\left( #1 \right)}
\newcommand{\ptg}[1]{\left\{ #1 \right\} }
\newcommand{\ptq}[1]{\left[ #1 \right] }
\newcommand{\abs}[1]{\left\lvert #1\right\rvert}
\newcommand{\norm}[1]{\left\lvert\left\lvert #1\right\rvert\right\rvert}
\newcommand{\giap}[1]{\langle  #1 \rangle}

\newcommand{\Ldwu}[1]{        \mathscr{L} \pt{     L^2_{w#1} }           }     
\newcommand{\Ldug}[1]{ \mathscr{L} \pt{  L^2 \pt{ \R_{#1} } } }

\title{Fourier Integral Operators of
\\ Boutet de Monvel Type} 
\author{U. Battisti, S. Coriasco and E. Schrohe}



\date{}

\begin{document}

\begin{abstract}
Given two compact manifolds with boundary $X,Y,$ and a boundary preserving 
symplectomorphism $\chi:T^*Y\setminus0\to T^*X\setminus0$, which is one-homogeneous
in the fibers and satisfies the transmission condition, we introduce Fourier integral operators 
of Boutet de Monvel type associated with $\chi$. 
We study their mapping properties between Sobolev spaces, 
develop a calculus and prove a Egorov type theorem. 
We also introduce a notion of ellipticity which implies the Fredholm property.
Finally, we show how -- in the  spirit of a classical construction by A. Weinstein -- 
a Fredholm operator of this type can be associated with $\chi$ and a section of the Maslov bundle.
If $\dim Y>2$ or the Maslov bundle is trivial, the index is independent of the section
and thus an invariant of the symplectomorphism. 
\end{abstract}

\maketitle

\section*{Introduction}
We develop a calculus of Fourier integral operators (FIOs) on compact manifolds with boundary, 
which extends the calculus of pseudodifferential boundary value problems defined by  Boutet de Monvel \cite{BU71}. 
Given two compact manifolds with boundary, $Y$ and $X$, we base our operators on symplectomorphisms $\chi: T^*Y \setminus 0 \to T^*X\setminus 0$, which are positively 
homogeneous of degree $1$ in the fibers. In case $X=Y$ and $\chi=id$ we recover the 
Boutet de Monvel calculus. 

Apart from the general interest in operators of this type, our main objective is to provide the analytic framework for an index problem in the spirit of A.\ Weinstein  \cite{WE75}. 
Weinstein considered two {\em closed} manifolds and a corresponding symplectomorphism between the cotangent bundles with the zero section removed.
He proved that this symplectomorphism defines, in a natural way, a FIO $F$ with the Fredholm property and that its index is a remarkable quantity: 
Let $X$ and $Y$ be additionally riemannian with Laplacians $\Delta_X$ and $\Delta_Y$ and suppose that the principal symbols satisfy $\sigma(\Delta_Y)=\sigma(\Delta_X)\circ \chi$. 
Denote by $\lambda_j(X)$ and $\lambda_j(Y)$ their  sequences of eigenvalues. Then, under a mild additional assumption (the Maslov class of $\chi$ has to be $\equiv 0$ mod $4$)
the sequence $|\lambda_j(Y)-\lambda_{j-\mathrm{ind}\,F}(X)|$ is bounded. 
Weinstein then asked for an expression of this index in topological terms.   See \cite{WE97} for on overview of the problem. 

In this article we show that, under suitable natural assumptions on the symplectomorphism $\chi$, 
we can associate with $\chi$ a
FIO in the calculus we develop, and establish the Fredholm property.

In the boundaryless case, C. Epstein and R. Melrose \cite{EM98} solved Weinstein's problem under the assumption that both manifolds coincide, relying on previous results by V. Guillemin \cite{GU84}, L. Boutet de Monvel \cite{BMG81}, and S. Zelditch \cite{ZE97}.
They reduced the task to the computation of the index of a Dirac operator on a closed manifold, constructed from the data, and thus to the Atiyah-Singer index theorem.  
This construction has been refined by C. Epstein in \cite{EPI}, \cite{EPII}, \cite{EPIII}. 
The general case, where $X$  and $Y$ are closed but possibly different, was treated by E. Leichtnam, R. Nest and B. Tsygan \cite{LNT01} in the framework of deformation quantization. 
A generalization to symplectomorphisms on a manifold with conical singularities 
has been studied by V. E. Naza{\u\i}kinski{\u\i}, B. Sternin and B.-W. Schulze \cite{NSS01}, \cite{NS06}, relying on the work of Epstein and Melrose. 
In the present article, however, we shall not tackle the problem of the computation of the index.

A second motivation for considering this class of FIOs is a theorem 
of  J.J. Duistermaat and I. Singer \cite{DS76}. They showed that -- under a mild topological condition -- every order-preserving isomorphism $i : L_{cl} (X) \to L_{cl}(Y)$ between the algebras of classical pseudodifferential operators on closed manifolds $X$ and $Y$, respectively, is of the form $i(A) = F^{-1}AF$, 
where $F$ is a FIO associated with a symplectomorphism as above. 
Recently, V. Mathai and R. Melrose \cite{MM12} found a proof which avoids the topological
condition. An analog of the Duistermaat-Singer theorem in the semiclassical setting has been given by H. Christianson \cite{CH11}.

We show in this article that conjugation with a FIO in our class provides an order-preserving isomorphism of Boutet de Monvel's algebra and 
we expect these to be all.  

The paper is organized as follows. In Section \ref{sec:trans}, we define the class of symplectomorphisms we work with. 
A variant of Moser's trick shows that a symplectomorphism $\chi$ as above can always be extended to a symplectomorphism 
$\tilde{\chi}: T^*\widetilde{Y}\setminus 0 \to T^*\widetilde{X}\setminus 0$, where
$\widetilde{Y}$ and $\widetilde{X}$ are neighborhoods of $X$ and $Y$, respectively, in larger closed manifolds.
The homogeneity, together with the fact  that it preserves the boundary, implies that $\chi$ induces a symplectomorphism 
$\chi_\partial: T^*\partial Y \setminus 0 \to T^*\partial X\setminus 0 $, 
which is the lift of a diffeomorphism $b: \partial Y \to \partial X$. 

We then analyze operators of the form $r^+ A^\chi e^+$, 
where $A^\chi$ is a (FIO) associated with $\tilde{\chi}$, 
$r^+$ is the restriction operator to $\operatorname{int} X$ and $e^+$ is the
extension-by-zero operator on functions in $Y$. 

As $\chi$ preserves the boundary, $r^+A^\chi e^+$ maps $C^\infty(Y)$ 
to $C^\infty(\operatorname{int}X)$.  
We require in addition that each component of $\chi$ satisfies the transmission condition.
This implies the continuity of 
$r^+A^\chi e^+: C^\infty(Y) \to C^\infty(X)$ and  $r^+(A^\chi)^*e^+:C^\infty(X)\to C^\infty(Y)$. 
In fact, it is also necessary, as we will show in a forthcoming paper \cite{BCS2}.

The assumptions on $\chi$ place the analysis here in a framework which is in a sense complementary to that considered by A. Hirschowitz and A. Piriou in \cite{HP77}.
They studied the transmission property for Fourier distributions conormal to hypersurfaces in $T^*X\setminus 0$. 

In Section \ref{sec:oscilint}, we establish the continuity properties of the above truncated FIOs 
$r^+ A^\chi e^+$, relying in a crucial way on the technique of operator-valued symbols. 
We prove that, for a symbol $a\in S^m(\R^n\times \R^n)$ satisfying the transmission condition,  
\[
  r^+\op^\psi_{n}(a)e^+: u \mapsto r^+\int e^{i \psi(x', x_n, \xi', \xi_n)-i \psi_\partial(x',\xi')} a(x', x_n, \xi', \xi_n) \widehat{e^+u}(\xi_n) \dbar \xi_n
 \]
 is an operator-valued symbol in $S^m \pt{\R^{n-1}, \R^{n-1}; \Si(\R_+), \Si(\R_+)}$. 
 Here $\psi$ is a phase function which locally represents $\chi$, while the
phase $\psi_\partial$ represents the symplectomorphism $\chi_\partial$. 
A key point is the analysis of $r^+ \op^\psi_n (a) \delta_0^{(j)}$, 
where $\delta_0^{(j)}$ is the $j$-th derivative of the Dirac distribution at the origin,  see Theorem 
\ref{teodirac}.
In contrast to the corresponding result in Boutet de Monvel's calculus, however, it is not true that $r^+ \op^\psi_n(a)e^+$ belongs to
$S^m \pt{ \R^{n-1}, \R^{n-1}; H^s \pt{\R_+}, H^{s-m} \pt{\R_+}}$
for each $s\in\R$ as we show in Remark \ref{rem:contro}. 
The section ends with the proof of the continuity of $r^+A^\chi e^+$ in the scale of Sobolev spaces.
 
In Section \ref{sec:bound} we complement the above truncated FIOs to matrices of Boutet de Monvel type operators of the form
 \[
 \mathcal{A}:=\left(
\begin{array}{cc}
r^+ A^\chi e^+ + G^{{\chi}_\partial} & K^{\chi_\partial}\\
 T^{\chi_{\partial}}& S^{\chi_\partial}
\end{array} \right)
\colon
\begin{array}{ccc}
C^\infty(Y,E_1)&&C^\infty(X,E_2)\\
\oplus &\to&\oplus\\
C^\infty(\partial Y,F_1)&&C^\infty(\partial X,F_2),
\end{array}\]
acting between sections of vector bundles $E_1$ over $Y$, $E_2$ over $X$, $F_1$ over
$\partial Y$ and $F_2$ over $\partial X$. 
Here, $G^{\chi_\partial}$, $K^{\chi_\partial}$, $T^{\chi_\partial}$, $S^{\chi_\partial}$
are FIOs with Lagrangian submanifold defined by the graph of $\chi_{\partial}$ and, respectively,
a singular Green symbol $g$
of order $m$ and type $d$, a potential symbol $k$ of  order $m$, a trace symbol $t$ of  order $m$ and type $d$,  a symbol $s \in S^{m}(\R^{n-1}\times \R^{n-1})$.

The set of such operators $\mathcal{A}$ is denoted by $\mathscr{B}_{\chi}^{m, d}(X \times Y)$.
We then develop the local version of a calculus which is -- under the usual restrictions --  closed under composition,  that is 
\[
\mathscr{B}_{\chi_1}^{m_1, d_1}\circ \mathscr{B}_{\chi_2}^{m_2, d_2}\subseteq \mathscr{B}_{\chi_1 \circ \chi_2}^{m_1+m_2, d}, \quad d= \max\{m_2+d_1, d_2\}.
\]
We conclude the section with the proof of a Egorov type theorem for this class of operators, see Theorem \ref{th:ego}.
 
In Section \ref{sec:prin}, the principal symbols of operators in $\mathscr{B}^{m, d}_\chi\pt{X\times Y}$ are introduced. 
Similarly as in Boutet de Monvel's calculus, the interior principal symbol $\sigma(\mathcal{A})$ is defined as the principal symbol of  $A^\chi$, restricted to $T^*Y\setminus0$.
The technique of $\mathscr{L}(\mathscr{S}(\mathbb R_+))$-valued symbols then enables us to identify also a homogeneous operator-valued boundary principal symbol $\sigma_\partial(\mathcal{A})$. 
Ellipticity, defined as the invertibility of both, then allows the construction of a parametrix in the calculus.

In Section \ref{sec:manif}, we extend the above local calculus to compact manifolds with boundary. One result we obtain is the following: 
Whenever $\mathcal{A} \in \mathscr{B}_\chi^{0,0}(X \times Y)$ is invertible,
\begin{equation}
\label{eq:iso}
j\colon  \mathscr{B}^{m,d}(X) \to \mathscr{B}^{m,d}(Y) 
\colon
\mathcal{P} \mapsto \mathcal{A}^{-1}\circ \mathcal{P}\circ \mathcal{A}
\end{equation}
is defined for all $m$ and $d$, hence extends to an isomorphism between algebras of Boutet de Monvel operators preserving the order, in the spirit of \cite{DS76}.

In the last Section \ref{sec:index} we show how an index can be associated with an admissible symplectomorphism $\chi$. 

To this end we first reduce to the case of a one by one matrix. Following the approach of A. Weinstein in \cite{WE75},
the natural candidate for a Fredholm operator associated with $\chi$ is
\[
\mathcal{U}=\pt{
  r^+ U^\chi e^+}
\]
where $U^\chi$ is a FIO defined by $\chi$ with a principal symbol $s$ which is a unitary section of the associated Maslov bundle.
We then establish the Fredholm property of this operator by showing the invertibility of 
$\sigma(\mathcal A)$ and $\sigma_\partial(\mathcal A)$. At this point the analysis is  more subtle than in the case of closed 
manifolds. A priori, it is not clear why the boundary symbol should be invertible. 
In order to show this, we use a deformation of the phase function via a scaling of the normal 
variable. In the limit, we obtain an invertible operator. However, as the phase is, in general,  discontinuous at the zero section, this is not a continuous deformation on $L^2(\mathbb R_+)$.
Instead, we work on the weighted space $L^2(\mathbb R_+, (1+x^2)^{-1}dx)$, where Schur's lemma implies the desired norm continuity. 
 
It turns out that the index of $\mathcal{U}$ is independent of the choice of the section $s$
whenever the Maslov bundle is trivial or $\dim Y > 2$. 

For the case $X=Y$, examples of admissible symplectomorphisms can be constructed by deforming 
the identity by means of a Hamiltonian flow. Of course, the index of the associated operator $\mathcal{U}$
will then be zero. In this respect, the situation is similar to the case of closed manifolds, where explicit examples of symplectomorphisms with nonzero index are lacking.
\bigskip

\textbf{Acknowledgements.}
Thanks are due to L.~Fatibene, A.~Fino and R.~Melrose for fruitful discussions, and to C. Epstein for explaining part of his work to us. 
We also want to express our special gratitude 
to R. Nest, with whom we worked on the proof of the Fredholm property. 
The first author has been supported by the Gruppo  Nazionale per l'Analisi Matematica, la Probabilit\`a e le loro  Applicazioni (GNAMPA) of the Istituto Nazionale di Alta Matematica (INdAM) and by the DAAD.

\section{Admissible Phase Functions}
\label{sec:trans}

The cotangent bundle $T^*Y$ of a manifold with boundary
 $\pt{Y, \partial Y}$ is  a symplectic manifold with boundary $T^*_{\partial Y}Y$. 
In this article we consider
two compact $n$-dimensional manifolds with boundary $\pt{Y, \partial Y}$ and  
$\pt{X, \partial X}$ and  a symplectomorphism 
\begin{eqnarray}\label{chi1}
\chi: T^*Y \setminus 0 \to T^*X \setminus 0
\end{eqnarray}
which is positively homogeneous of degree one in the fibers. 
We require that $\chi$ preserves  the boundary, that is
\begin{eqnarray}\label{chi2}
\chi\pt{ \partial T^*Y \setminus 0} = \partial T^*X\setminus 0.
\end{eqnarray}
The following lemma,  which is proven in \cite{ME81}, analyzes symplectomorphisms of this type. 
\begin{Lem}
  \label{lem:chidelta}
Under assumptions \eqref{chi1} and \eqref{chi2}, $\chi$  induces
 a symplectomorphism 
\begin{eqnarray}\label{chid}
\chi_{\partial}: T^*\partial Y \setminus 0 \to T^*\partial X\setminus 0,
\end{eqnarray}
  positively homogeneous of degree one in the fibers, such that the following diagram commutes:
  \begin{align*}
    \xymatrix{
	    &T^*_{\partial Y}Y \setminus N^*{\partial Y} \ar@{^(->}[d]^{i^*_Y}  \ar[r]^{\chi}  & T^*_{\partial X}X
	    \setminus N^*{\partial X}  \ar@{^(->}[d]^{i^*_X} \\
	    &T^*\partial Y \setminus 0 \ar[r]^{\chi_{\partial}}  & T^*\partial X \setminus 0.
	    }
  \end{align*}
\end{Lem}
\begin{rem}
  \label{rem:triv}
  In Lemma \ref{lem:chidelta} we have considered the induced symplectomorphism $\chi_{\partial}$
  outside the zero section. Actually, 
  since $\chi$ is smooth on $\partial T^*Y \setminus 0$, the induced symplectomorphism $\chi_{\partial}$
is also smooth at the zero section. Since $\chi_{\partial}$ is positively homogeneous of degree one in the fibers, 
  the smoothness at the zero section implies that $\chi_{\partial}$ is trivial in the fibers. That is, $\chi_{\partial}$ is the lift of a diffeomorphism 
  \begin{eqnarray}
b: \partial Y \to \partial X
\end{eqnarray}
of the boundaries  (see \cite{DAS01}). 
\end{rem}

The manifolds  $X$ and $Y$ embed into closed manifolds of the same dimension. 
Moser's trick, see \cite[Ch. 7]{DAS01}, then allows us to extend $\chi$ to a symplectomorphism
\begin{eqnarray}
\widetilde \chi: T^*\widetilde Y\setminus 0 \to T^*\widetilde X\setminus 0,
\end{eqnarray}
positively homogeneous of degree one in the fibers 
over neighborhoods $\widetilde X$ of $X$ and $\widetilde Y$ of $Y$ in these closed manifolds.

It will be important to understand the form of the Jacobian of $\chi$ in a collar neighborhood of the boundary. We write
\begin{align}
  \nonumber
  &\chi  (y', y_n, \eta', \eta_n)\\
  &\ \ =(x'^*(y', y_n, \eta', \eta_n), x_n^*(y', y_n, \eta', \eta_n), \xi'^*(y', y_n, \eta', \eta_n), \xi_n^*(y', y_n, \eta', \eta_n)).\label{eq:compsympl}
\end{align}
We suppose that the coordinates $(y', y_n, \eta', \eta_n)$, $(x'^*, x_n^*, \xi'^*, \xi_n^*)$ determine a collar neighborhood of the boundary,
$y_n, x_n^*$ being boundary defining functions. Since the boundary is preserved, $x^*_n(y', 0, \eta', \eta_n)=0$ for all $(y', \eta', \eta_n)$. 
Hence, $\partial_{y'}x_n^*$, $\partial_{\eta'}x_n^*$, $\partial_{\eta_n} x_n^*$ are identically zero at $y_n=0$. Moreover, 
Lemma \ref{lem:chidelta} implies that the restrictions $x'^*_\partial$ and $\xi'^*_\partial$ of
$x'^*$ and $\xi'^*$ to $y_n=0$ locally define the symplectomorphism $\chi_\partial$. 
These restrictions are then independent
of the conormal direction, that is $\partial_{\eta_n}x'^*$ and $\partial_{\eta_n}\xi'^*$ are identically zero at the boundary. 
Hence, we find that the Jacobian of $\chi$ at the boundary has the  form
\begin{align}
\label{eq:jacob}
 J(\chi)_{T_{\partial Y}^*Y}=\left(
\left.  \begin{array}{cccc}
    \partial_{y'}x'^* & \partial_{\eta'}x'^* & \partial_{y_n}x'^*& 0 \\
    \partial_{y'}\xi'^* & \partial_{\eta'}\xi'^* & \partial_{y_n}\xi'^*& 0 \\
    0 & 0 & \partial_{y_n}x_n^*& 0 \\
    \partial_{y'}\xi_n^* & \partial_{\eta'}\xi_n^* & \partial_{y_n}\xi_n^*& \partial_{\eta_n}\xi_n^* 
  \end{array} \right)\right|_{y_n=0}.
\end{align}
Moreover, since
\begin{align}
 \label{eq:simplbound}
 J_{\chi_\partial}=
 \left(
 \begin{array}{cc}
  \partial_{y'}x'^*_\partial & \partial_{\eta'}x'^*_\partial\\
  \partial_{y'}\xi'^*_\partial & \partial_{\eta'}\xi'^*_\partial 
 \end{array}
 \right)
\end{align}
is a symplectic matrix, it has determinant $1$.
Clearly also $J(\chi)_{T_{\partial Y}^*Y}$ has determinant equal to $1$, 
because $\chi$ is a symplectomorphism.
This implies that $\partial_{y_n}x_n^*|_{y_n=0} \cdot \partial_{\eta_n}\xi_n^*|_{y_n=0}=1$. In particular, $\partial_{y_n}x_n^*|_{y_n=0}$ and $\partial_{\eta_n}\xi_n^*|_{y_n=0}$ 
can never vanish. 
As the boundary is compact, they are bounded away from zero. 
We even see that  $\partial_{y_n}x_n^*|_{y_n=0}>0$, since 
$x_n^*$ and $y_n$ are boundary defining functions.

We recall a well-known property of Lagrangian subspaces, which extends to the case of manifolds with boundary.
\begin{Prop}
  \label{prop:loclag}
  Let $\Lambda \subset T^*Z\setminus 0 $ be a conic Lagrangian submanifold. Then, for each $\lambda_0=(z_0,\theta_0) 
  \in \Lambda$, there exists a neighborhood $U$ of $z_0$ and a phase function $\phi$ defined in a conic neighborhood
  $U \times \Gamma\subseteq U \times \R^N$ -- $N$  large enough -- such that $\phi$ parametrizes $\Lambda$ in a conic neighborhood
  of $\lambda$. That is
 \begin{align*}
    C_{\phi}=\{(z, \theta) \mid \phi'_\theta(z, \theta)=0\} &\to T^*Z\\
    (z, \theta) &\mapsto (z, \phi'_z(z, \theta))
 \end{align*}
 induces a diffeomorphism in $U\times\Gamma$.

If $\Lambda$ is locally defined by the graph of a symplectomorphism, 
we have a splitting of the variable $z$ as $z=(x,y)$, and  
$\Lambda \subseteq T^*X \times T^*Y$. 
In this case, we can choose a phase function of the particular form
  $\phi(x, y, \theta)=\psi(x, \theta)- y \cdot \theta$, with $\phi \in C^\infty(\Omega_{x_0} \times \Omega_{y_0} \times
  \Gamma)$, where $\Omega_{x_0}$ and $\Omega_{y_0}$ are neighborhoods of $x_0$ and $y_0$, respectively, $\Gamma$ is a cone in $\R^n\setminus 0$,
  $2n$ is the dimension of $\Lambda$.
\end{Prop}

\begin{rem}
  \label{rem:psipartial}
We can apply Proposition \ref{prop:loclag} also to the symplectomorphism $\chi_\partial$
in order to obtain a phase function $\phi_\partial(x', y', \theta')=\psi_\partial (x', \theta') - y'\cdot \theta'$ which represents $\chi_\partial$.
  Since $\chi_\partial$ is the lift of a diffeomorphism, the phase function $\psi_\partial (x',\theta')$ is smooth at $\theta'=0$ and therefore linear in $\theta'$.
\end{rem}

We will not recall here the notion of Maslov bundle, see \cite{HO03} for its precise description.

\begin{Lem}
  \label{lem:mas}
Under the above assumptions  the Maslov bundle of 
  \begin{align}\label{Lambda}
    \Lambda=\textnormal{Graph}(\chi)'= \{(x, \xi), (y, -\eta)\mid \chi(y, \eta)=(x, \xi)\} \subseteq T^*X\setminus 0 \times T^*Y \setminus 0
  \end{align}
  is trivial in a collar neighborhood of $\partial \Lambda= (\partial T^*X \times \partial T^*Y) \cap \Lambda$.
\end{Lem}

\begin{proof}
  Let $\lambda_0=(x_0,y_0,\xi_0,\eta_0) \in \Lambda$, $U^\Lambda_{\lambda_0}$ a conic neighborhood of 
  $\lambda_0$ and
  $\phi(x,y,\theta)$ $=\psi(x,\theta)-y\cdot\theta$ a phase function representing $\Lambda$ on
  $U^\Lambda_{\lambda_0}$. 
  Then the Maslov bundle is trivialized on $U^\Lambda_{\lambda_0}$
  by $e^{\frac{i}{4} \textnormal{sgn} \;\phi_{\theta}''}$. 
  By Remark \ref{rem:psipartial}, $\chi_\partial$ is linear. This implies that $\phi''_{\theta'}$
  vanishes identically at the boundary 
  for each covector $\theta'$ tangent to the boundary. Moreover, $\partial_{\theta_n}\phi$
  is constant at the boundary, hence $ \partial^2_{\theta_n}\phi$   vanishes identically at the boundary. 
  Therefore, $e^{\frac{i}{4} \textnormal{sgn} \;\phi''_\theta}=1$ identically in a small neighborhood of $\lambda_0$. 
  As this holds for all $\lambda_0 \in \partial \Lambda$, the result follows from the compactness of the
  boundary. 
\end{proof}

In general it is not possible to find a global phase function defining the whole Lagrangian 
submanifold. In fact, this is impossible whenever the Maslov bundle is not trivial, 
see \cite{LSV94} for the precise statement. 
In our setting, the triviality of the Maslov bundle implies the following:

\begin{Prop}
  \label{prop:global}
For every $x_0 \in \partial X$ there exist neighborhoods  $U_{x_0}$ and $U_{y_0}$,
of $x_0$ and $y_0=b^{-1}(x_0)$, respectively,
  such that $\chi(T^* U_{y_0}) \subseteq T^*U_{x_0}$. Moreover
  it is possible to define phase functions
  \begin{equation}
    \label{eq:condfase}
    \begin{split}
      &\phi_L(x,y, \eta)= \psi_L(x, \eta) - y \cdot \eta\\
      &\phi_R(x,y, \xi)= x\cdot \xi - \psi_R(y, \xi)
    \end{split}
  \end{equation}
which parametrize $\chi$ in $\Lambda \cap\pt{ T^* U_{x_0} \times T^*U_{y_0}}$, where
$\Lambda$ is given by \eqref{Lambda}.
\end{Prop}
\begin{proof}
We write
  \begin{align*}
    &\chi: 
     (y, \eta)\mapsto (x^*(y, \eta), \xi^*(y, \eta))
    \intertext{and}
    &\chi^{-1}:     (x, \xi)\mapsto (y^*(x, \xi), \eta^*(x, \xi)).
  \end{align*}
  Since the symplectomorphism preserves the boundary it is possible to find  neighborhoods
  $U_{x_0}, U_{y_0}$ such that $\chi(T^* U_{y_0}) \subseteq T^*U_{x_0}$.
Here we write $(x, \xi)=(x', x_n, \xi', \xi_n)$, $(y, \eta)=(y',y_n, \eta', \eta_n)$ with boundary defining functions $x_n$ and $y_n$. 
In view of the considerations around \eqref{eq:jacob}, we can suppose - possibly restricting $U_{y_0}$ - that 
  \begin{equation}\label{eq:nondegxi}
    \det\pt{\partial_{\eta}\xi^*(y,\eta)}\neq 0 \text{ on $T^*U_{y_0}$},
  \end{equation}
  and  
  \begin{equation}
    \label{eq:nondegeta}
    \det \pt{ \partial_{\xi} \eta^*(x, \xi)}\neq 0 \text{ on $T^*U_{x_0}$}.
  \end{equation}
  Following the idea of \cite{LSV94} we  introduce
   \begin{align*}
    \widetilde{\psi}_L (x, y, \xi)&=(y^*(x, \xi)- y)\cdot \eta^*(x, \xi)\\
    \widetilde{\psi}_R (x, y, \eta)&=(x- x^*(y, \eta))\cdot \xi^*(y, \eta).
   \end{align*} 
  Since $\chi$ and $\chi^{-1}$ preserve the canonical $1$-form we have that
  \begin{align*}
    &\xi^* \cdot x^*_{\eta_k}=0, \qquad \eta^*\cdot y^*_{\xi_k}=0,\\
    & \xi^* \cdot x^*_{y_k}=\eta_k, \qquad \eta^*\cdot y^*_{x_k}=\xi_k. 
  \end{align*}
  The above relations together with the non degeneracy conditions \eqref{eq:nondegxi}, \eqref{eq:nondegeta} imply
  that $\widetilde{\psi}_L$ and $\widetilde{\psi}_R $ are phase functions representing 
  $\Lambda \cap \pt{ T^* U_{x_0} \times T^*U_{y_0}}$.
  In order to have the phase function as in \eqref{eq:condfase},
  we use the inverse mapping theorem. In fact, by \eqref{eq:nondegeta} and \eqref{eq:nondegxi} it is possible to invert 
  $\xi^*(y, \eta)$ and $\eta^*(x, \xi)$ in $T^*U_{y_0}$ and $T^*U_{x_0}$ respectively. We denote by
  $\tilde{\eta}(y, \xi)$ and $\tilde{\xi}(x, \eta)$ the inverse functions of $\xi^*(y, \eta)$ and $\eta^*(x, \xi)$, respectively.
  We then set
  \begin{align*}
      \phi_L\pt{x,y, \eta}&= (y^*(x, \tilde{\xi}(x, \eta))- y)\cdot \eta\\
      \phi_R\pt{x,y, \xi}&= (x- x^*(y, \tilde{\eta}(y, \xi)))\cdot \xi 
  \end{align*}
  and we obtain the assertion.
\end{proof}

In order to define a suitable calculus for FIOs on manifolds with boundary, we need to introduce
the transmission
condition, see, e.g., \cite{BU71,GR87,GH90,RS85,SC01}. Consider the function spaces:
\begin{align*}
  H^+= \ptg{\mathscr{F}(e^+u)\mid u \in  \Si(\R_+)} \quad \mbox{ and } \quad H^-_0=\ptg{\mathscr{F}(e^-u)\mid u \in \Si(\R_-)},
\end{align*}
where $\Si(\R_\pm)= r^\pm \Si(\R)$ is the restriction of Schwartz  functions on $\R$ to the right (left) half line,
and $e^\pm$ is the extension by zero of a function defined on $\R_\pm$.
It is easy to prove that the functions in $H^+$ and $H^-_0$ decay to first order at infinity.
 Moreover, we denote by  $H'$ the set of all polynomials in one variable. Then we  define
\begin{align*}
  H= H^+ \oplus H^-_0 \oplus H'.
\end{align*}
\begin{Def}
  \label{Def:trans}
  Let $a \in S^m(\R^n \times \R^n \times \R^n)$. Then
  $a$ satisfies  the transmission condition  at $x_n=y_n=0$ provided that, for all 
  $k,l$,
  \[
    \partial^k_{y_n} \partial_{x_n}^l a(x', 0, y', 0, \xi', \langle \xi' \rangle \xi_n) \in S^m(\R^{n-1} \times 
    \R^{n-1} \times \R^{n-1}) \hat{\otimes}_{\pi} H_{\xi_n}.
  \]
  We denote by $S^m_{\tr}(\R^{n} \times \R^n \times \R^n)$
  the subset of symbols satisfying the transmission condition.
\end{Def}

For symbols positively homogeneous of degree $m$ in $\xi$ for large $|\xi|$, 
Definition \ref{Def:trans} is equivalent to 
\begin{equation}
  \label{eq:hom}
  \partial_{x_n}^k \partial_{y_n}^l\partial_{\xi'}^\alpha \partial_{x'}^\beta a\pt{x', 0, y', 0, 0,1}=
  (-1)^{m-|\alpha|} \partial_{x_n}^k \partial_{y_n}^l\partial_{\xi'}^\alpha \partial_{x'}^\beta a\pt{x', 0, y', 0,0, -1}
\end{equation}
for all $k, l\in \N$, $\alpha, \beta  \in \N^{n-1}$. The above condition is often called symmetry condition.
The proof of the equivalence can be found in \cite{RS85}. 

\begin{Def}[Admissible symplectomorphism] 
 \label{rem:adm} 
We call a symplectomorphism $\chi$ as above admissible, if 
  all the components of $\chi$ locally satisfy the transmission condition at the boundary. A phase function that represents an admissible 
  symplectomorphism will be called admissible. 
\end{Def}
\begin{rem}
Definition \ref{rem:adm} has an invariant meaning, because a change of coordinates 
in the cotangent bundle, induced by a change of coordinates in the base manifold,
is linear with respect to the fibers.
  Hence, if the transmission condition is satisfied in one local chart then it is satisfied also in any other.
  \end{rem}
\begin{ex}
 The easiest way to construct an admissible symplectomorphism is to consider $X=Y$
 and a Hamiltonian flow generated by a function $f \in C^{\infty}(T^*X \setminus 0)$ 
 which is positively homogeneous of degree $1$ in the fibers,
 satisfies the transmission conditions with respect to $\partial X$ and 
has vanishing  normal derivative  at $T^*_{\partial X}X \setminus 0$. 
\end{ex}
\section{Oscillatory Integrals}
\label{sec:oscilint}
In this section we will analyze the continuity properties of oscillatory integrals 
arising from FIOs associated with Lagrangian submanifolds obtained from admissible symplectomorphisms  as in Definition \ref{rem:adm}.
We will use the concept of operator-valued pseudodifferential operators acting on 
weighted Sobolev spaces over $\R_+$; see the 
Appendix for basic definitions and results. 

Let us consider $\widetilde{A} \in I^m_{comp} ( \widetilde{X} \times\widetilde{Y}, \widetilde{\Lambda})$, where $\widetilde{\Lambda}= 
\textnormal{Graph}(\widetilde{\chi})'$.  
The definition of FIOs implies that for all $( x_0, y_0, \xi_0, \eta_0)=
\lambda_0 \in \widetilde{\Lambda}$, the operator is  microlocally given by a kernel of the type
\begin{align}
  \label{eq:oscil}
  \int e^{i \phi(x',x_n,y', y_n, \xi',\xi_n )}a(x', x_n, y', y_n,\xi', \xi_n) \dbar \xi' \dbar \xi_n,
\end{align}
up to smooth kernels. We will focus on the situation where boundary points are involved, so 
we suppose that, at $(x_0, y_0) \in \widetilde{X} \times \widetilde{Y}$, the local coordinates 
$(x_1, \ldots, x_n, y_1, \ldots,$ $y_n)=(x',x_n,y',y_n)$ are chosen so that $x_n$, $y_n$ are boundary
defining functions. We also identify the chart domains on $\widetilde{X}$ and $\widetilde{Y}$ 
with the corresponding open subsets
$\Omega_x,\Omega_y\subset\R^n$. As the  Lagrangian submanifold $\widetilde{\Lambda}$ is defined by the
graph of a  symplectomorphism,  we can always assume that 
\begin{align*}
  \phi \pt{x', x_n, y', y_n, \xi', \xi_n}= \psi\pt{x', x_n, \xi', \xi_n}- i y'\cdot \xi'- i y_n \cdot \xi_n. 
\end{align*}
The phase function $\phi$ is (initially) defined in an open conic neighborhood 
$\Gamma$ in $\Omega_x \times \Omega_y \times \pt{\R^n \setminus \{0\}}$, and
the symbol $a(x, y,\xi)$ has support contained in $\Gamma$. We set
\[
	\Omega^\partial_x=\Omega_x\cap\{x_n=0\},\quad
	\Omega^\partial_y=\Omega_y\cap\{y_n=0\},	
\]
and
\begin{align}
\label{eq:omega+}
	\Omega^+_x=\Omega_x\cap\{x_n\ge0\},\quad
	\Omega^+_y=\Omega_y\cap\{y_n\ge0\}.	
\end{align}
We also recall that, since the FIO is associated with a symplectomorphism, 
we can rely on representations both
by left and right quantization, see \cite[Ch. 25]{HO04}. 

For convenience, we will proceed under the following technical assumptions.
\begin{Ass}
	\label{TechnHyp:1}
	\begin{itemize}
	\item 
	The amplitude $a$ satisfies the transmission condition w.r.t. 
	$x_n=0$, $y_n=0$. As $\chi$ is admissible, this is preserved under 
	changes of coordinates.
	\item  $\psi$ is defined on  $\R^n \times( \R^n \setminus \{0\})$.
	In fact, since $a$ vanishes outside $ \Gamma$, we can choose any \emph{good} 
	extension for $\psi$; see \cite{BCS14} for the extension that we will use below.
	%
	\item Since the kernel \eqref{eq:oscil} represents an operator on a compact manifold,
	$a$ can be assumed 
	to vanish unless  $|x_n|$ and $|y_n|$ are small, or $x^\prime$, $y^\prime$ 
	lie outside suitable compact subsets of $\R^{n-1}$.
	Moreover, since the kernels \eqref{eq:oscil} are given modulo smoothing operators,
	it is no restriction to assume also that $a$ vanishes for $\xi$ in a neighborhood of 
	the origin. Otherwise, we can insert a $0$-excision function in the amplitude, which changes \eqref{eq:oscil} by a smooth kernel.
\end{itemize}
\end{Ass}
	 
Assumptions \ref{TechnHyp:1} allow us to focus on
oscillatory integrals of the type
\begin{align}
  \label{eq:osciint}
  \int e^{i\psi \pt{x', x_n, \xi', \xi_n} - iy' \cdot \xi'- i y_n \cdot \xi_n }
  a\pt{x',x_n, y', y_n,\xi', \xi_n} \dbar \xi' \dbar\xi_n,
\end{align}
with $\psi$ and $a$ as above.

\medskip

We will next analyze the action in the normal direction
of an operator with kernel as in \eqref{eq:osciint}. Before, however, we introduce a class of functions which will be useful in the sequel.
\begin{Def}
  \label{def:BS}
  A function $a \in C^{\infty}(\R^{n-1}_{x'}  \times \R_{x_n} \times \R^{n-1}_{\xi'}\times \R_{\xi_n})$ belongs to the set
  $BS^{m} \pt{\R^{n-1}, \R^{n-1}; S^{l}\pt{\R}}$ 
  if, for all $\alpha, \beta \in \N^{n-1}$,
  \begin{align*}
    (x_n, \xi_n)\mapsto \partial_{\xi'}^\alpha \partial_{x'}^\beta a \pt{x', \frac{x_n}{ \giap{\xi'} }, \xi', \xi_n \giap{ \xi'} }  \in S^{l}(\R \times \R)
  \end{align*}
  and each seminorm can be estimated uniformly by $\langle \xi' \rangle^{m-|\alpha|}$. That is, for all  $\gamma, \delta \in \N$
  there exists
  a constant $C_{\gamma, \delta}$ such that
  \begin{align*}
    \abs{
    \partial_{\xi_n}^{\gamma}\partial_{x_n}^{\delta} \ptq{\partial_{\xi'}^\alpha \partial_{x'}^\beta a\pt{x', \frac{x_n}{\giap{\xi'}}, \xi', \xi_n \giap{\xi'} }
    }}
    \leq C_{\gamma, \delta} 
    \giap{ \xi_n}^{l-|\gamma|} \giap{\xi'}^{m-|\alpha|}, \quad x, \xi \in \R^{n}.
  \end{align*}
\end{Def}
Definition \ref{def:BS} implies that
$a \in BS^m \pt{ \R^{n-1}, \R^{n-1}, S^m\pt{\R}}$, if $a \in S^m(\R^n\times \R^n)$. 
This a consequence of the fact that
$\partial_{\xi'}^{\alpha}\partial_{x'}^{\beta}  a \in S^{m-|\alpha|} \pt{ \R^n\times \R^n}$
and the estimate
\begin{align*}
 \abs{ 
 \partial_{\xi_n}^\gamma \partial_{x_n}^\delta\ptq{ a \pt{ x', \frac{x_n}{\giap{ \xi'}},
 \xi', \xi_n \giap{\xi'} }  }
 }
 \leq
 C \giap{ \xi'}^{m} \giap{\xi_n}^{m- |\gamma|}, \quad x,\xi \in \R^{n}.
\end{align*}
 
Moreover, it is clear that $BS$-spaces satisfy a multiplicative property, that is
\begin{align*} 
 BS^{m} \pt{\R^{n-1}, \R^{n-1}; S^l \pt{\R}} \cdot BS^{m'} \pt{\R^{n-1}, \R^{n-1}; S^{l'}\pt{\R}} &\\
 \label{eq:prodbs}
 \subseteq BS^{m+m'} &\pt{\R^{n-1}, \R^{n-1}, S^{l+l'}\pt{\R}}.
\end{align*}
The proof of the Lemma \ref{Lem:induc} and Theorem \ref{thm:genint}, below, can be found in \cite{BCS14}.
\begin{Lem}
  \label{Lem:induc}
  Let $a \in S^m(\R^n \times \R^n)$ and $\psi$ satisfy Assumptions \ref{TechnHyp:1}.
Then
 \begin{equation}
    \label{induc}
     \begin{aligned}
    &\partial_{\xi'}^{\alpha}\partial_{x'}^\beta \pt{e^{i \psi(x', x_n, \xi', \xi_n)-i \psi_{\partial} 
    (x', \xi')}a \pt{x', x_n,\xi', \xi_n}}\\ 
    &\hspace{2.5cm}= e^{i \psi(x', x_n, \xi', \xi_n)- i\psi_{\partial}(x', \xi')} \tilde{a} \pt{x', x_n, \xi', \xi_n},
  \end{aligned}
  \end{equation}
  where $\tilde{a}
  \in BS^{m-|\alpha|} \pt{\R^{n-1}, \R^{n-1}; S^{m+|\beta|}\pt{\R}}$.
\end{Lem}

\begin{rem}
	\label{rem:phfprop}
	Since $\chi$ preserves the boundary,
	$\partial_{\xi_n} \psi(x', 0, \xi', \xi_n)$ is identically equal to zero,
hence $\psi(x', 0, \xi', \xi_n)$ is independent of $\xi_n$. We set
\begin{align}
  \label{eq:psidelta}
  \psi \pt{x', 0, \xi', \xi_n}=\psi_\partial\pt{x', \xi'} 
\end{align}
and notice that $\psi_\partial$  represents the symplectomorphism at the boundary $\chi_\partial$, described in Remark \ref{rem:triv}. 
\end{rem}

\begin{Thm}
  \label{thm:genint}
  Let $a \in S^m_{\textrm{tr}}\pt{\R^{n}\times \R^{n}}$ and $\psi$ be as in Assumptions \ref{TechnHyp:1}. Then, the operator
  \begin{align*}
    \op^\psi_{n}(a): u \mapsto \int e^{i \psi(x', x_n, \xi', \xi_n)-i \psi_\partial(x',\xi')} a(x', x_n, \xi', \xi_n) \hat{u}(\xi_n) \dbar \xi_n
  \end{align*}
  belongs to $S^m \pt{\R^{n-1}, \R^{n-1}; \Si(\R), \Si\pt{\R}}$.
\end{Thm}

\begin{rem}
 Theorem \ref{thm:genint} is also valid for $\pt{ \op^\psi_{n}(a) }^t $, hence 
 $\op^{\psi}_{x_n}(a)$ can be extended to an operator-valued symbol in $S^m \pt{\R^{n-1}, \R^{n-1}; \Si'(\R), \Si'\pt{\R}}$.
As   $\iota:L^2(\R) \to \Si'(\R)$ and $e^+:L^2(\R_+)\to L^2(\R)$ can both be interpreted 
 as  operator-valued symbols of order $0$, we also have
 \begin{align}
  \label{eq:contl2b}
  \op^\psi_{n}(a) e^+ \in S^m \pt{\R^{n-1}, \R^{n-1}; L^2(\R_+), \Si'\pt{\R}}.
 \end{align}
 Moreover, since the proof of Theorem \ref{thm:genint} is based on the theory of SG FIOs, 
 see \cite{CO99b,CO99}, it is also possible to prove that for every $(x', \xi')$
 \begin{align}
  \label{eq:contsob}
  \op^\psi_{n}(a) e^+ \in \mathcal{L}(L^2(\R_+), H^{-m}(\R)).
 \end{align}
In general, however, we can not replace $\Si'(\R_+)$ by $H^{-m}(\R_+)$ in 
  \eqref{eq:contl2b}, see Remark \ref{rem:contro}.
\end{rem}

The next step is to consider the action of an oscillatory integral as in \eqref{eq:osciint}
on derivatives of Dirac distributions at the origin.  

\begin{Thm}
  \label{teodirac}
  Let $a \in S^m_{\textrm{tr}} \pt{\R^{n}\times \R^{n}}$ and $\psi$ be as in Assumptions \ref{TechnHyp:1}. Then
  \begin{eqnarray*}
    k_j(x', \xi')&=& r^+ \op^{\psi}_n(a) \delta^{(j)}_0\\
    &=&r^+\int e^{i\psi(x', x_n, \xi', \xi_n)- i \psi_{\partial}(x', \xi')}a(x',x_n, \xi', \xi_n) \hat{\delta}^{(j)}_{0}(\xi_n) \dbar\xi_n 
  \end{eqnarray*}
  defines an operator-valued symbol in $S^{m+ \frac{1}{2}+j} \pt{\R^{n-1}, \R^{n-1}; \C, \Si\pt{\R_+}}$. Here,
  $\delta^{(j)}_0$ is the $j$-th derivative of the Dirac distribution at $0$.
\end{Thm}
\begin{proof}
 We start by considering the operator $\op^\psi(a)$ acting on smooth functions defined on the whole of $\R^n$.
  \[
    \begin{split}
    \op^\psi(a):\; C_c^{\infty}&(\R^n) \to C^\infty(\R^n)\\
      u &\mapsto \int e^{i \psi(x',x_n,\xi',\xi_n)}a(x',x_n, \xi', \xi_n) \hat{u}(\xi',\xi_n)\, \dbar \xi' \dbar\xi_n\\
      &=\int e^{i \psi_\partial(x',\xi')} \int e^{i r(x',x_n, \xi', \xi_n)}a(x',x_n, \xi', \xi_n) \hat{u}(\xi',\xi_n)\, \dbar \xi' \dbar \xi_n,
    \end{split}
  \]
  where $r(x',x_n, \xi', \xi_n)= \psi(x', x_n, \xi', \xi_n)- \psi_\partial(x', \xi')$.
 First let $j=0$. Going over to the right quantization, we find
  \begin{eqnarray}    
\lefteqn{\hspace*{0cm}\op^\psi(a)(\phi \otimes \delta_0) (x_n)} \nonumber\\
    \nonumber
    &=& \int e^{i \psi_{\partial} (x',\xi')}   \int e^{i \psi(x', x_n, \xi', \xi_n)- \psi_{\partial}(x', \xi')}
    a(x',x_n, \xi', \xi_n) \hat{\delta}_{0}(\xi_n)\, \dbar\xi_n \hat{\phi} (\xi')\dbar \xi'\\
    \nonumber
    &=&\int e^{i x' \cdot \xi' + i x_n \cdot \xi_n - i \psi^{-1}(y',y_n, \xi', \xi_n) } a_R(y',y_n, \xi', \xi_n) 
    \phi(y')\otimes \delta_0(y_n)\, dy' dy_n \dbar \xi' \dbar \xi_n\\
    \nonumber
    &=&\int e^{i x' \cdot \xi' - i \psi^{-1}_{\partial}(y',\xi')} \int e^{i x_n \cdot \xi_n -
    i r^{-1} (y',y_n, \xi', \xi_n)} a_R(y',y_n, \xi', \xi_n)\cdot \\
    \label{eq:oppsi}
    && \hspace{3.8cm}\cdot\phi(y')\otimes \delta_0(y_n)\, dy' dy_n \dbar \xi' \dbar \xi_n,
  \end{eqnarray}
  where the equality is modulo operators with smooth kernel. In \eqref{eq:oppsi},
  \begin{align*}
    r^{-1}(y',y_n, \xi', \xi_n) = \psi^{-1}(y',y_n, \xi', \xi_n)- \psi^{-1}_{\partial}(y',\xi'),
  \end{align*}
  $\psi^{-1}$ is the phase function representing the symplectomorphism $\chi^{-1}$,
  and $\psi_\partial^{-1}= \psi^{-1}|_{y_n=0}$.
  Now, we focus on the action in the normal direction, namely, the expression
  \begin{align*}
    B \pt{y', \xi', x_n}\pt{\delta_0}= \int e^{i x_n \cdot\xi_n } \int
    e^{- i r^{-1}(y', y_n, \xi', \xi_n)}a_R \pt{y',y_n,\xi',\xi_n} \delta_0(y_n)\, dy_n \dbar\xi_n.
  \end{align*}
  The symbol $a_R$ satisfies the transmission condition, so we can write
  \[
    a_R \pt{y', 0, \xi', \xi_n \langle \xi' \rangle}=
    \sum_{k=0}^m s_k^R \pt{y', \xi'}\xi_n^k \giap{ \xi'}^k+ \sum_{l=0}^\infty \lambda_l b_l^R \pt{y', \xi'}
    \hat{h}_l(\xi_n),
  \]
  where $s^R_k \in S^{m-k} \pt{\R^{n-1}\times \R^{n-1}}$, $\ptg{\lambda_l}_{l \in \N} \in l^1$, $\ptg{b^R_l}_{l \in \N} \subset S^m \pt{\R^{n-1}\times \R^{n-1}}$
  is a null sequence, $\ptg{h_l}_{l \in \N}$ a null sequence in $ \Si(\R_{+}) \oplus \Si(\R_{-})$.
  By the definition of operators on distributions, we have for all $u \in C_c^\infty(\R)$
  \begin{eqnarray}
 \lefteqn{\langle \kappa_{\giap{\xi'}^{-1} }
    B \delta_0, u \rangle= \langle \delta_0, B^t  \pt{\kappa_{\giap{\xi'}} u\ }\rangle} \nonumber\\
    &=&\langle \delta_0, \giap{ \xi'}^{\frac{1}{2}} \int e^{- i r^{-1}(y', x_n, \xi', \xi_n)+ i y_n \cdot \xi_n} 
    a_R(y', x_n, \xi', \xi_n) u(\giap{\xi'} y_n)dy_n \dbar \xi_n\rangle  \nonumber\\ 
    &=& \giap{\xi'}^{\frac{1}{2}} \int a_R \pt{y',0, \xi', \xi_n \giap{ \xi'}} 
    \hat{u}(- \xi_n)\dbar \xi_n   \nonumber\\
    &=& \giap{\xi'}^{\frac{1}{2}} \sum_{k=0}^m  s^R_{k} (y', \xi') \int \xi_n^k  \giap{ \xi'}^k \hat{u} (-\xi_n) \dbar \xi_n \label{eq:ar0}\\
    &&\qquad \qquad +\giap{\xi'}^{\frac{1}{2}}
    \sum_{l=0}^\infty \lambda_l b^R_l \pt{y',\xi'}  \int \hat{h}_l(\xi_n) \hat{u}(-\xi_n) \dbar \xi_n.
    \nonumber
  \end{eqnarray}
  Using the properties of the Fourier transform,
  \begin{align}
    \nonumber
    \langle \kappa_{\langle \xi' \rangle^{-1} } B \delta_0, u \rangle=&\langle \xi' \rangle^{\frac{1}{2}}
    \sum_{k=0}^m s^R_{k}(y', \xi') \langle \xi'\rangle^k
    (-1)^k i^k \giap{\delta_0^{(k)},u}  \\ \label{eq:dirac45}
    &+\giap{\xi'}^{\frac{1}{2}}\sum_{l=0}^\infty \lambda_l b^R_l(y',\xi') \int {h}_l(x_n) {u}(x_n) dx_n.
  \end{align}
  Applying the restriction operator $r^+$, all terms that depend on $\delta_0^{(k)}$ vanish, so we get
  \begin{equation}
    \label{eq:rpiu}
    \kappa_{\giap{\xi'}^{-1} } r^+ \pt{B  \delta_0}(y',x_n, \xi') = \giap{\xi'}^{\frac{1}{2}}\sum_{l=0}^\infty \lambda_l b_l^R(y',\xi') r^+h_l(x_n).
  \end{equation}
 Derivatives w.r.t.\ $(x', \xi')$ can be treated in the same way. Hence  $(r^+ B \delta_0) \left(y', \xi'\right) \in S^{m+ \frac{1}{2}}(\R^{n-1}, \R^{n-1}; \C, \Si \pt{\R_+})$. 
  Inserting \eqref{eq:rpiu} 
  into \eqref{eq:oppsi}, we obtain 
  \begin{multline*}
    r^+ \op^\psi(a) (\phi \otimes \delta_0)(x_n)=
     \int e^{i x'\cdot \xi' - i \psi_\partial^{-1}(y',\xi')}(r^+ B\delta_0) \pt{y', \xi'} \phi(y')dy' \dbar \xi'\\
    =  \langle \xi' \rangle \sum_{l=0}^\infty r^+ h_l( \langle \xi' \rangle x_n)  \int e^{i x'\cdot \xi' - i \psi_\partial^{-1}(y',\xi')} 
    \lambda_l b^R_l(y',\xi') 
    \phi(y')dy' \dbar \xi'.
  \end{multline*}
Switching back to the left quantization, we obtain, modulo smoothing operators, the symbol-kernel
  \[
    \kappa_{\giap{\xi'}} r^+ \op_n^\psi(a)\delta_0= \giap{\xi'}^{\frac{1}{2}} \sum_{l=0}^\infty 
    \lambda_l b_l (x',\xi') r^+h_l (x_n),
  \]
  with a suitable null sequence $\{b_l\}\subset   S^m (\R^{n-1}\times \R^{n-1})$.
  This implies the assertion for $j=0$.
  The proof for $j>0$ is similar. In fact, it is enough to notice that
  \begin{align}
    \label{eq:derdirac}
    \hat\delta_0^{(j)}(\xi_n)=  (i\xi_n)^j \hat\delta_0,
  \end{align}
 so we can follow the same steps, but with a symbol of order $m+j$. 
 
 Finally, we have to take into account all the seminorms of $\Si(\R_+)$, hence to consider derivatives with respect to the
 $x_n$-variable. Lemma \ref{Lem:induc} implies that this step can be obtained from the previous one using a different symbol,
 which still satisfies Assumptions \ref{TechnHyp:1}, since the phase is admissible. 
\end{proof}
\begin{rem}
  \label{dirac-}
  With the same notation as in Theorem \ref{teodirac}, we obtain that
  \begin{align*}
    r^- \int e^{i \psi(x',x_n, \xi', \xi_n)- i \psi_\partial(x', \xi')}a(x', x_n, \xi', \xi_n) \hat{\delta}_0^{(j)} \dbar \xi_n
  \end{align*}
  is a symbol in $S^{m+ \frac{1}{2}+j} (\R^{n-1}, \R^{n-1}; \C, \Si\pt{\R_-})$.
\end{rem}

\begin{rem} \label{rem:diracreg}
The proof of Theorem \ref{teodirac}, in particular \eqref{eq:ar0}, shows more.
Suppose that the symbol $a$  vanishes at $x_n=0$ to order $\ge m_++1$, where 
$m_+=\max\{m,0\}$.
Then $a_R(x',0,\xi)$ has no polynomial part.  Hence 
\begin{align}
    \label{eq:equival}
    e^+r^+ \op^{\psi}(a) \delta_0 = \chi_{\R_+}\op^{\psi}(a)\delta_0,
  \end{align}
where $\chi_{\R_+}$ is the characteristic function of $\R_+$.

If $a$ vanishes even to order $\ge m_++N+1$ for some $N\in\N$, then Equation 
\eqref{eq:equival} also holds with $\delta_0$ replaced by $\delta_0^{(j)}$ for 
$j=0,\ldots, N$. 
 \end{rem}
\begin{rem}
  \label{rem:diractop}
We noted in the proof of Theorem \ref{teodirac} that  
$r^+\op_n^\psi(a)\delta^{(j)}_0= r^+\op_n^\psi(a \xi_n^j)\delta_0$, $j>0$.
If we then switch to the right quantization, we obtain, instead of the powers of $\xi_n$, 
derivatives of $a_R$ with respect to $y_n$ and factors of $y_n$-derivatives of $\psi^{-1}$.
 The top order term will involve only the $j$-th power of $\partial_{y_n}\psi^{-1}$,
 and no terms containing higher order derivatives of the phase.
 Finally, evaluating this at the boundary, we find an expression depending only on the right symbol at the boundary, multiplied by  $(\partial_{y_n}\psi^{-1}|_{y_n=0})^j$.  
This remark will be crucial in the definition of the principal symbol in Section \ref{sec:prin}, more precisely in Theorem \eqref{th:simbprin}.
\end{rem}

The action on the Dirac distributions at the origin is the key point to prove the continuity of 
the operator $r^+ \op_n^{\psi}(a) e^+$ 
in $\Si(\R_+)$. Before, we need  a technical lemma, whose proof we omit, since it is straightforward. 
\begin{Lem}
 \label{lem:omega1}
 Let $\zeta$ be an excision function of the origin, that is $0 \leq \zeta \leq 1$, $\zeta(t)=0$ 
 for $|t|\leq 1$ and
 $\zeta(t)=1$ for $|t|\geq 2$. 
 Then, the function $\xi_n\mapsto\zeta\left(\frac{\xi_n}{\langle \xi' \rangle}\right) \xi_n^{-l}$ belongs to $S^{-l}(\R\times\R)$.
\end{Lem}

\begin{Thm}
  \label{th:conts}
  Let $a$ and $\psi$ satisfy Assumptions \ref{TechnHyp:1}. Then
  \begin{align*}
    r^+ \op_n^\psi(a) e^+ : \:& \Si \pt{\R_+} \to \Si \pt{\R_+}\\
    & u\mapsto r^+ \iint e^{i \psi(x', x_n, \xi', \xi_n)- i \psi_{\partial}(x', \xi')- i y_n \cdot \xi_n} a(x',x_n, \xi', \xi_n) e^+ u (y_n)dy_n \dbar \xi_n 
  \end{align*}
  is an operator-valued symbol in $S^m \pt{\R^{n-1}, \R^{n-1}; \Si(\R_+), \Si\pt{\R_+}}$. 
  Corresponding results hold for $r^-\op^\psi_n(a) e^+$ and $r^+ \op^\psi_n(a)e^-$.
\end{Thm}
\begin{proof}
  Choose $\zeta$ as in Lemma \ref{lem:omega1} and write
  \begin{align*}
    &r^+ \op^\psi_n(a)(e^+ u) =r^+A_1(e^+u)+ r^+A_2(e^+u),
  \end{align*}
  with 
  \begin{align*}
    &A_1= \op^\psi_n \pt{\zeta\pt{\frac{\xi_n}{\giap{\xi'}} }a }
    \text{ and }
    A_2=\op^\psi_n \pt{\pt{1-\zeta\pt{\frac{\xi_n}{\giap{\xi'}}} }a }.
  \end{align*}
Lemma \ref{lem:omega1} implies that both $A_1$ and $A_2$ 
  satisfy the hypotheses of Theorem \ref{teodirac}. Moreover, the kernel of the operator $A_2$ is smooth with Schwartz decay, due to 
  the compact support of the symbol in the $\xi_n$ variable and in the $x_n$ variable.  Each $\Si(\R_+)$-seminorm of $r^+A_2 e^+u$ is 
  bounded by $\giap{\xi'}^m$, because the symbol belongs to $S^m \pt{\R^n\times \R^n}$. Therefore $r^+ A_2 e^+$ belongs to $S^m\pt{\R^{n-1}, \R^{n-1}; \Si(\R_+), \Si\pt{\R_+}}$.
  We can write
  \begin{equation}
    \label{eq:contsa1}
    r^+ A_1 e^+ u= \int e^{i \psi(x', x_n, \xi', \xi_n)- i \psi_{\partial}(x', \xi')} \zeta \pt{\frac{\xi_n}{\giap{\xi'}
    }}a(x',x_n, \xi', \xi_n)\frac{1}{(i \xi_n)^l} (i\xi_n)^l \widehat{e^+ u} (\xi_n) \dbar \xi_n.
  \end{equation}
  The properties of the Fourier transform assure that $i \xi_n \widehat{e^+u}(\xi_n)=
  \widehat{e^+\partial u}+u(0) \widehat{\delta_0}$.
  Theorem \ref{teodirac} implies that the part of the integral depending on derivatives of Dirac distributions satisfies the desired bound. 
  To analyze the other part and the derivatives with respect to $x', \xi', x_n$, we note that, by taking $l$ large
  enough, we are left with an integral operator in the normal direction, for which the $\sup$ norm of the kernel
  can be estimated in terms of the seminorms of $a$ times $\giap{\xi'}^m$. The details are left to the reader.
\end{proof}
\begin{rem}
  \label{rem:contro}
The proof of Theorem \ref{th:conts} indicates that  the order of the operator in the normal direction may increase if we differentiate w.r.t. the $x'$-variable.
In contrast to the usual Boutet de Monvel calculus, it is not true in general that
  \begin{align*}
    &r^+ \op_n^\psi(a) e^+ \in S^m \pt{\R^{n-1}, \R^{n-1}; H^{s} \pt{\R_+}, H^{s-m} \pt{\R_+}}. 
  \end{align*}
  This can be seen explicitly by means of the following 
  example. Define
  \begin{align}
    \nonumber
    A:& \;\Si(\R^n) \to \Si(\R^n)\\
    \label{eq:contro}
    &u \mapsto \iint e^{i[(x'-y')\cdot \eta'+ (f(x')x_n-y_n)\cdot \eta_n]} u(y', y_n)dy' dy_n \dbar \eta' \dbar\eta_n,
   \end{align}
  \noindent where $f$ is a strictly positive function. The phase function of the FIO $A$ 
  in \eqref{eq:contro} represents 
  a symplectomorphism
  $\chi$ of $T^*\overline{\R^n}_+=\overline{\R^{2n}_+}$ onto itself of the  form
  \begin{align}
    \nonumber
    \chi: &  \;\overline{\R^{2n}}_+ \to \overline{\R^{2n}}_+\\
    \label{eq:simple}
    & (y', y_n, \eta', \eta_n)\mapsto \left(y', f(y')^{-1}y_n, \eta'+ f'(y') y_n \frac{\eta_n}{f(y')}, f(y') \eta_n \right).
  \end{align}
  The symplectomorphism $\chi$ is admissible since it preserves the boundary:
  \[
    \chi(y', y_n, \eta', \eta_n)\in \partial \overline{\R^{n}}_+ \Leftrightarrow (y', y_n, \eta', \eta_n) \in \partial 
    \overline{\R^{n}}_+, \mbox{ that is }y_n=0.
  \]
  It is linear in the fibers, therefore all components have the transmission property.
  Looking at the action along the normal direction, we see that \eqref{eq:contro} cannot be extended to an
  operator-valued symbol in $S^{0}\pt{\R^{n-1}, \R^{n-1}; H^{s}\pt{\R}, H^{s}\pt{\R}}$. Indeed,
  \begin{align*}
    &  \kappa_{\giap{\eta'}^{-1}}\partial_{x'_j}\iint  e^{i (f(x')x_n-y_n)\cdot \eta_n}  
    \kappa_{\giap{\eta'}} u( y_n)  dy_n \dbar \eta_n\\
    &=\kappa_{\langle \eta'\rangle^{-1}} \left( \langle \eta'\rangle^{\frac{1}{2}} \iint e^{i (f(x')x_n-y_n)\cdot \eta_n} 
    i (\partial_{x'_j} f)(x')x_n \eta_n u(\langle \eta' \rangle y_n) dy_n \dbar \eta_n \right)\\
    &=\kappa_{\langle \eta'\rangle^{-1}} \left( \langle \eta'\rangle^{-\frac{1}{2}} \iint e^{ i(\langle \eta' \rangle
    f(x')x_n-z_n)\cdot \frac{\eta_n}{\langle \eta'\rangle}} 
    i (\partial_{x'_j} f)(x')x_n \eta_n u(z_n) dz_n \dbar \eta_n \right)\\
    &=\kappa_{\langle \eta'\rangle^{-1}} \left( \langle \eta'\rangle^{-\frac{1}{2}} \int e^{ i
    f(x')x_n\cdot \eta_n} 
    i (\partial_{x'_j} f)(x')x_n \eta_n \hat{u}\left(\frac{\eta_n}{\langle \eta' \rangle}\right) \dbar \eta_n \right)\\
    &=\kappa_{\langle \eta'\rangle^{-1}} \left( \langle \eta'\rangle^{\frac{1}{2}} \int e^{ i\langle \eta' \rangle
    f(x')x_n\cdot \theta_n} i (\partial_{x'_j} f)(x')x_n \theta_n \langle \eta' \rangle \hat{u}(\theta_n) \dbar \theta_n \right)\\
    &=\kappa_{\langle \eta'\rangle^{-1}} \left( \langle \eta'\rangle^{\frac{3}{2}} \int e^{ i\langle \eta' \rangle
    f(x')x_n\cdot \theta_n}  (\partial_{x'_j} f)(x')x_n
    \widehat{\partial_{x_n} u}(\theta_n) \dbar \theta_n \right)\\
    &=\kappa_{\langle \eta'\rangle^{-1}} \left(\langle \eta'\rangle^{\frac{3}{2}}   (\partial_{x'_j} f)(x')
    x_n \pt{\partial_{x_n}u}(\langle \eta' \rangle f(x')x_n)\right)\\
    &=  \pt{\partial_{x'_j}f}(x') x_n \pt{\partial_{x_n} u}(f(x')x_n).
  \end{align*}
  Hence we lose one derivative and the operator-valued symbol does not send $H^s(\R)$
  into $H^s(\R)$.
\end{rem}
\noindent Now, we recall a technical lemma, proven in \cite{RS85}, p. 122.
\begin{Lem}
  \label{lem:div}
  Let $a \in S^m(\R^n\times \R^n)$ be a symbol with the transmission property.
  Then there exists a symbol $a_1 \in S^m(\R^n\times \R^n)$ having the transmission property for all hyperplanes 
  $x_n=\epsilon$, $\epsilon \geq 0$, such that
  \[
    \partial_{x_n}^k \pt{a \pt{x', x_n, \xi', \xi_n}-a_1 \pt{x', x_n, \xi', \xi_n}}|_{x_n=0}=0
  \]
  for all $k \in \N$ and for all $x', \xi', \xi_n$.
  The symbol $a_1$ has the following expression
  \begin{equation}
    \label{eq:ser3}
    a_1(x', x_n, \xi', \xi_n)= \sum_{j=0}^\infty \frac{x_n^j}{j!}\partial_{x_n}^j a(x',0, \xi', \xi_n) \phi(t_j x_n),
  \end{equation}
  where $\phi$ is a cut-off function at the origin and $\{t_j\}$ is a sequence in $\R_+$ such that the series in \eqref{eq:ser3}
  converges in $S^m\pt{\R^n\times\R^n}$.
\end{Lem}
\begin{Prop}
  \label{prop:div}
  Let $a$ and $\psi$ satisfy Assumptions \ref{TechnHyp:1}. Then, it is possible to write
  \[
    r^+ \op_n^\psi(a) e^+= r^+ \op_n^\psi(a_d)e^+  + r^+ \op_n^\psi(a_0)e^+,
  \]
  with $a_0$ such that 
  \begin{equation}
    \label{eq:taglio}
    \pt{\op_n^{\psi}(a_0) e^+}u \in L^2(\R), \quad  u \in \Si(\R_+),
  \end{equation}
  and $a_d \in S^m_{\tr}(\R^n\times \R^n)$ is a polynomial in  $\xi_n$.
\end{Prop}
\begin{proof}
  The proof follows from Remark \ref{rem:diracreg} and an  observation in the proof of 
  Theorem \ref{th:conts}.
  Choose $a, a_1$ as in Lemma \ref{lem:div} and set $b= a-a_ 1$. In view of
  the transmission property of $a$ we can write
  \[
    \partial_{x_n}^ja(x', 0, \xi', \xi_n) =\sum_{k=0}^{m}a_{k,j}(x', \xi') \xi_n^k+ 
    \sum_{k=0}^{\infty} \lambda_{k, j} b_{k,j}(x', \xi') h_{k,j}\left(\frac{\xi_n} 
    {\langle \xi' \rangle}\right),
  \]
  with $a_{k,j} \in S^{m-k}(\R^{n-1}, \R^{n-1})$, $b_{k,j} \in S^{m}(\R^{n-1}\times \R^{n-1})$, $h_{k,j} \in H^+\oplus H^-_0$,
  $(\lambda_{k,j})_k \in l^1$,
  and
  \begin{align*}
   &b_{k,j}, h_{k,j}\rightarrow 0, \quad k\to \infty.
  \end{align*}

  Then, we set
  \begin{align*}
    a_d(x', x_n, \xi', \xi_n)=& \sum_{j=0}^\infty \sum_{k=0}^{m} \frac{x_n^j}{j!} a_{k,j}(x', \xi') \xi_n^k \phi(t_j x_n),\\
    a_0^1(x', x_n, \xi', \xi_n)=& \sum_{j=0}^\infty \sum_{k=0}^\infty \frac{x_n^j}{j!} \lambda_{k,j} b_{k,j}(x', \xi') h_{k,j} \left(\frac{\xi_n} 
    {\langle \xi' \rangle}\right)\phi(t_j x_n),\\
    a_0=& a_0^1+b.
  \end{align*}
  By construction, $a= a_d+ a_0$. Notice  that $a_0^1$ 
is a symbol of order zero w.r.t. the $\xi_n$-variable. We then conclude from  
\eqref{eq:contsob} with $m=0$ that $\op_n^\psi(a_0^1)e^+u$ belongs to $L^2(\R)$.
  To analyze $\op_n^\psi(b)e^+u$ we use the scheme in Theorem \ref{th:conts}. 
We split the operator into two parts: one smoothing and the other with a symbol vanishing to infinite order at $\xi_n=0$.
  For the smoothing part, \eqref{eq:taglio} holds. 
In the  other part, we divide and multiply by $\xi_n^l$ as in \eqref{eq:contsa1}, 
with $l$ arbitrary. 
We obtain a symbol which is of order $m-l$ in the $\xi_n$ variable and derivatives of Dirac's delta  up to the order $l-1$. To handle these terms, we use Remark \ref{rem:diracreg}, valid for symbols vanishing to infinite order 
  at the origin. Namely, 
  \begin{equation}
    \op^{\psi}_n\pt{b(x', x_n, \xi', \xi_n)\zeta\pt{\frac{\xi_n}{\giap{\xi'}}}  (i \xi_n)^{-l}} \delta^{(j)}_0 \in L^2(\R), \quad  j\in \N,
  \end{equation}
   concluding the proof.
\end{proof}
\begin{rem}
  \label{rem:tras}
  We have proven that $r^+ \op^\psi_n(a)e^+$ 
  is a continuous operator from $\Si(\R_+)$ to itself, so it is possible to define the transposed operator 
  \[ 
    \pt{r^+ \op^\psi_n(a) e^+}^t.
  \] 
  It is important to stress that, in general,
  \begin{equation}
    \label{eq:tras}
    \pt{r^+\op^\psi_n(a)e^+}^t u \not= 
    r^+ \pt{\op^\psi_n(a)}^te^+  \, u, \quad u \in \Si(\R_+).
  \end{equation}
  A simple counterexample is the operator $\pt{r^+ \partial e^+}^t$, since
  \[
    \pt{r^+ \partial_{x_n} e^+}^t u= 
    -r^+\partial_{x_n} e^+ u- 
    u(0) \delta_0.
  \]
  Nevertheless, if $a_0$ is as in Proposition \ref{prop:div} and therefore \eqref{eq:taglio} is fulfilled, 
  then equality holds in \eqref{eq:tras}. Indeed, for $u,f\in\Si(\R_+)$,
  \begin{align}
    \nonumber	
    \langle ( r^+ \op^\psi_n(a_0) e^+ )^t u, f\rangle&=
    \langle u, r^+ \op^\psi_n(a_0) e^+ f\rangle= \langle e^+ u,  \op^\psi_n(a_0) e^+ f\rangle\\
   \label{eq:trasposto}
    &=\langle \op^\psi_n(a_0)^t e^+ u, e^+f\rangle=\langle r^+ (\op_n^\psi a_0)^t e^+ u, f\rangle.
  \end{align} 
Theorem \ref{th:conts} then implies that 
  $r^+ \op^\psi_n \pt{a_0}e^+\in S^m \pt{\R^{n-1}, \R^{n-1}; \Si'\pt{\R_+}, \Si'\pt{\R_+}}$.  
  Since $C_c^{\infty}(\R_+)$ is dense in $\Si'(\R_+)$, 
  we can also define the action on $\Si'(\R_+)$ by 
  \[
    (r^+ \op_n^\psi(a_0) e^+) u=
    \lim_{k\to \infty} 
    r^+ \op^\psi_n(a_0) e^+u_k, \quad u_k \to u \mbox{ in } \Si'(\R_+). 
  \]
 
\end{rem}
\begin{Lem}
  \label{limdelta}
  Let $u \in \Si'(\R)$ be smooth on $\R_+$. Then the following statements are equivalent:
  \begin{itemize}
    \item[i)] For all $j \in \N$
    \[
      \lim_{x\to 0^+}\partial^j u(x)=c^j,\quad c^j \in \C.
    \]
    In particular the function $r^+u$ can be extended  smoothly up to zero.
    \item[ii)]
    For all $j \in \N$ and for all sequences $\ptg{\psi_m^j}_{m\in \N} \subseteq C_c^{\infty} \pt{\R_+}$ such that
    \begin{equation}
      \label{diracj}
      \psi_m^j \to (-1)^{j}\delta_0^{(j)} \quad \mbox{in } \Si'(\R),
    \end{equation}
    we have
    \[
      \lim_{m\to \infty}\langle u, \psi^j_m\rangle=c^j, \quad c^j \in \C.
    \]
    There is a trivial continuous inclusion $i: C_c^\infty(\R_+) \to C_c^\infty(\R)$ given by the extension by zero, so 
    the limit \eqref{diracj} is well defined.
  \end{itemize}
\end{Lem}
The proof is left to the reader. The implication \emph{i)} $\Rightarrow$ \emph{ii)} is almost trivial.
For the converse one can argue by contradiction.

\begin{Thm}
  \label{teosmen}
  Let $\psi$ and $a$ satisfy Assumptions \ref{TechnHyp:1}. By Proposition \ref{prop:div} we 
  can write $a= a_d+ a_0$, where $\op_n^\psi \pt{a_0}e^+$ maps $\Si \pt{\R_+}$ to $L^2 \pt{\R}$. Hence
  $ \pt{e^+r^+-1}\op^\psi_n \pt{a_0}e^+= - e^- r^-\op_n^\psi \pt{a_0}e^+$. Moreover, $r^- \op^\psi_n \pt{a_0}e^+$
  extends to an operator
  \begin{equation}
    \label{eq:teosmeneq}
    r^- \op^\psi_n \pt{a_0}e^+: \Si'\pt{\R_+} \to \Si\pt{\R_-}
  \end{equation}
  and defines a symbol in $S^{m} \pt{\R^{n-1}, \R^{n-1}; \Si' \pt{\R_{+}}, \Si \pt{\R_{-}}}$.
 \end{Thm}
\noindent Notice that $e^+$ in \eqref{eq:teosmeneq} is not defined on $\Si'(\R_+)$. 
The operator $r^- \op^\psi_n(a_0)e^+$ is defined as 
 the dual of $r^+ \op^\psi_n(a_0)^t e^-\,: \Si(\R_-) \to \Si(\R_+)$.
 
\begin{proof}
  We want to prove that for all  $s_1, s_2$, for all $\delta, \gamma$,  we have
  \[
    p_{\delta,\gamma} \pt{\kappa_{\giap{\xi'}^{-1}} r^- \op^\psi_n(a_0) e^+\kappa_{\giap{\xi'}} u} 
    \prec \giap{\xi'}^{m} \|u\|_{H^{-s_1, -s_2}_0}, 
    \quad u \in C_c^\infty \pt{\R_+},
  \]
  $\ptg{p_{\delta,\gamma}}$ being the seminorms of $\Si(\R_-)$.
  By Assumptions \ref{TechnHyp:1}, $a$ vanishes for $|x_n|>\epsilon$,
  $\epsilon$ small. As $\partial_{\xi_n}\psi(x', x_n, \xi)\neq 0$
  if $x_n \neq 0$,
  the phase function has no critical points on $\supp a_0$ outside $\ptg{x_n=0}$. 
  An integration by parts argument shows that $\singsupp \op_n^\psi(a_0) u \subseteq \ptg{x_n=0}$.
  Hence, we only need to consider the behavior as $x_n \to 0^-$.
  From Theorem \ref{th:conts} and Proposition \ref{prop:div}, we recall that the following maps are
  continuous:
  \begin{equation}
    \label{contSi}
    r^- \op_n^\psi(a_0) e^+: \Si(\R_{+}) \to \Si(\R_{-})
  \end{equation} 
  and
  \begin{equation}
    \label{contSii}
    r^-\op^\psi_n(a_0) e^+ : \Si'(\R_{+}) \to \Si'(\R_{-}).
  \end{equation}
  In order to prove that $r^- \op^\psi_n (a_0)e^+$ belongs to
  $S^m \pt{ \R^{n-1}, \R^{n-1}; \Si'\pt{\R_+},\Si \pt{\R_-}}$, we have to analyze
  \begin{equation}
    \label{semnorm}
    \lim_{x_n \to 0^-} \partial^{k}_{x_n} \kappa_{\giap{\xi'}^{-1}} \pt{ \partial_{x'}^\beta \partial_{\xi'}^\alpha 
    r^-\op^\psi_n(a_0) e^+ \kappa_{\giap{\xi'}} u}(x_n) .
  \end{equation}
  We start with the case $|\alpha|=|\beta|=0$. 
  By definition,
  we have to prove that, for all $s_1, s_2, k$ we have
  \[
    \abs{\lim_{x_n \to 0^-}  \partial^{k}_{x_n} \kappa_{\giap{\xi'}^{-1}}
    \pt{  r^-\op^\psi_n(a_0) e^+ \kappa_{\giap{\xi'}} u}(x_n) }
    \prec   \giap{\xi'}^{m} \|u\|_{H^{-s_1, -s_2}_0(\R)},\quad  u \in C_c^\infty(\R_+).
  \]
  Using the idea of Lemma \ref{limdelta}, we consider 
  \[
    \langle \kappa_{\giap{\xi'}^{-1}}  r^-\op^\psi_m(a_0) e^+ \kappa_{\giap{\xi'}}u, (-1)^k \partial_{x_n}^k \psi_l \rangle, 
  \]
  where $ \ptg{\psi_l}_{l \in \N} \subseteq C_c^{\infty}\pt{\R_{-}}$  is a sequence such that
  \begin{equation}
    \label{condir}
    \psi_{l} \to \delta_0, \quad \mbox{in } \Si'(\R).
  \end{equation}
  Notice that \eqref{condir} implies that $\kappa_{\giap{\xi'}} \psi_l $ converges to 
  $\giap{\xi'}^{-\frac{1}{2}} \delta_0$. 
  By Remark \ref{rem:tras} and Proposition \ref{prop:div} we have
  \begin{align*}
    \langle\kappa_{\giap{\xi'}^{-1}} 
    r^-\op^\psi_n(a_0) \kappa_{\giap{\xi'}} e^+ u,  \psi_l \rangle
    =\langle   u , \kappa_{\giap{\xi'}^{-1}} r^+\pt{
    \op^\psi_n(a_0)^t  \kappa_{\giap{\xi'}}} e^-\psi_l \rangle.
  \end{align*}
  By \eqref{contSii} and \eqref{condir} we get
  \begin{align*}
    \lim_{l \to \infty} \langle  u, \kappa_{\giap{\xi'}^{-1}}
    r^+ \op^\psi_n(a_0)^t \kappa_{\giap{\xi'}} e^-\psi_l  
     \rangle
    =\langle \kappa_{\giap{\xi'}^{-1}}
    u, r^+\op^\psi_n(a_0)^t 
    \kappa_{\giap{\xi'}} \delta_0\rangle.
  \end{align*}
  By Theorem \ref{teodirac} we know that 
  \[
    r^+\pt{\op_n^\psi(a_0)^t} \delta_0\in S^{m+\frac{1}{2}} \pt{\R^{n-1}, \R^{n-1}; \C, \Si\pt{\R_+}},
  \]
  so, finally, since $\Si(\R_-)= \proj_s H^\s(\R_-),$
  \begin{align*}
    &\hspace{-.3cm}\lim_{l\to \infty}  \abs{\langle  
    u ,  \kappa_{\giap{\xi'}^{-1}} r^+
    \op^\psi_n(a_0)^t \kappa_{\giap{\xi'}} e^-\psi_l\rangle}\\
    &\leq  \| \kappa_{\langle\xi'\rangle^{-1}} r^+
    \op_n^\psi(a_0)^t
    \kappa_{\langle\xi'\rangle}\delta_0\|_{H^{s_1, s_2}(\R_+)} \|u\|_{H^{-s_1, -s_2}_0(\overline \R_-)}
     \leq C \langle\xi'\rangle^{m} \|u\|_{H^{-s_1, -s_2}_0(\overline \R_-)}.
  \end{align*}
  For the derivatives w.r.t. $\xi'$ and $x'$ of orders $\alpha$ and $\beta$, we use 
  Lemma \ref{Lem:induc} and a slight variation of Theorem \ref{teodirac} in the setting of symbols belonging to the class
  $BS^{m-|\alpha|}\pt{\R^{n-1}, \R^{n-1};S^{m+|\beta|} (\R)}$.
 \end{proof}

\begin{Thm}
  \label{th:contsobolev}
  Let $a$ and $\psi$ satisfy Assumptions \ref{TechnHyp:1}. Then
  \[
    r^+\op^\psi (a) e^+: H^{s}(\R^n_+)\to H^{s-m}(\R^n_+), \quad s>-\frac{1}{2},
  \]
  continuously.
\end{Thm}
\begin{proof}
  For $s\leq 0$ the result follows from the continuity of $e^+\colon H^s(\R^n_+)$ $\to H^s(\R^n)$, $r^+\colon H^s(\R^n)\to H^s(\R^n_+) $ and the continuity properties of FIOs with homogeneous phase.
  It remains to consider the case $s>0$. 
  Using interpolation we may assume $s \in \N$. We write
  \[
    r^+\op^\psi(a) e^+=  r^+ \op^\psi(a) e^+ \circ \Lambda_+^{-s} \circ \Lambda_+^{s},
  \]
  where $\Lambda_+^s= r^+ \Lambda^s e^+$ is a truncated   pseudodifferential operator in the sense of Boutet de Monvel
  such that 
  $\Lambda_+^s: H^s(\R^n_+) \cong L^2(\R^n_+)$,  and $\Lambda_+^{-s}$ is the inverse of $\Lambda_+^s$. 
  So, we only need to prove that $r^+ \op^\psi(a) e^+ \circ \Lambda_+^{-s}: L^2(\R_+^n)\to H^{s-m}(\R^n_+)$ 
  is continuous. We observe that
  \begin{equation}
    \label{conts}
    r^+\op^\psi(a) e^+ \circ r^+ \Lambda^{-s}_+ e^+= r^+ \op^\psi(a) \circ \Lambda^{-s} e^+- 
    r^+ \op^\psi(a)\pt{e^+r^+-1}\Lambda^{-s}_+e^+.
  \end{equation}
  The operator $\op^\psi(a) \circ \Lambda^{-s}$, by the properties of FIOs is, modulo operators with smoothing kernel,
  a FIO of order $m-s$ with phase $\psi$.
  Thus, $r^+ \op^\psi(a) \circ \Lambda^{-s} e^+: L^2(\R^n_+) \to H^{s-m}(\R^n_+)$ is continuous,
  since $e^+$ is continuous on $L^2$. Now, we have to analyze
  the second term of \eqref{conts}. We treat it as a  FIO defined on the boundary 
  with operator-valued symbol.  Notice that $\Lambda^{-s}_+$ is of negative order, 
  and the differential part of the decomposition in Proposition \ref{prop:div} vanishes, so
  \[
    r^+ \op_n^\psi(a) (e^+ r^+ -1) \Lambda^{-s} e^+ u= - r^+ \op^\psi_n(a) e^- r^- \Lambda^{-s} e^+ u, \quad u 
    \in C_c^\infty(\R_+).
  \]
  According to the general theory of
  Boutet de Monvel's calculus, $r^- \Lambda_+^{-s} e^+$ extends to a symbol in
  $S^{-s} \pt{ \R^{n-1}, \R^{n-1}; \Si'\pt{\R_+}, \Si\pt{\R_-}}$; by Theorem \ref{th:conts}, we know 
  that 
  \[
    r^+ \op_n^\psi(a)e^- \in S^{m} \pt{\R^{n-1}, \R^{-1}; \Si \pt{\R_-}, \Si\pt{\R_+}}.
  \] 
  So, $r^+ \op_n^\psi (a) e^- 
  r^- \Lambda^{-s}_+e^-$ is a symbol in $S^{m-s}$ $(\R^{n-1} $ $, \R^{n-1}; \Si'(\R_+), \Si(\R_+))$. 
  We can therefore write
  $r^+ \op^\psi_n(a) (e^+ r^+-1)\Lambda^{-s}_+e^+$ as an operator-valued FIO 
  defined on the boundary with phase function $\psi_{\partial}$ and an amplitude 
  belonging to $S^{m-s}(\R^{n-1}, \R^{n-1};$ $\Si'(\R_+), \Si(\R_+))$. The continuity of
  operator-valued pseudodifferential operators on wedge Sobolev spaces implies that
  \[
    \xymatrix{
    & L^2(\R^n_+)  \ar@{^(->}[r]  & \mathscr{W}^{0}( \R^{n-1}; L^2(\R_+))\ar[d]^{r^+ \op^\psi(e^+r^+-1)\Lambda^{-s}r^+} \\
    & H^{s-m}(\R^n_+)  & \ar@{^(->}[l] \mathscr{W}^{s-m}(\R^{n-1}; \Si(\R_+)) , }
  \]
  where $\mathscr{W}^s(\R^{n-1}; E)$ denotes the wedge Sobolev space of order $s$ with values in 
  the topological vector space $E$, see the Appendix \ref{sec:append}.
 \end{proof}

\section{Fourier Integral Operators of Boutet de Monvel Type}
\label{sec:bound}
We recall the definition of three symbol classes in the Boutet de Monvel calculus.
  \begin{itemize}
    \item [i)] A \emph{potential symbol} of order $m$ is an element of
    \[
      S^{m}(\R^{n-1}, \R^{n-1}; \C, \Si(\R_+))= \proj_\mathbf{s}S^m(\R^{n-1}, \R^{n-1}; \C, H^\mathbf{s}(\R_+)).
    \]
    \item[ii)] A \emph{trace symbol} of order $m$ and type zero is an element of the set
    \[
      S^m(\R^{n-1}, \R^{n-1}; \Si'(\R_+), \C)=\proj_\mathbf{s}S^m(\R^{n-1}, \R^{n-1}; H^\mathbf{s}_0
      (\overline{\R}_+), \C);
    \]
    it also  defines  a symbol 
    in $S^{m}(\R^{n-1}, \R^{n-1};H^{s_1, s_2}(\R_+), \C)$, if $s_1>-\frac{1}{2}$. 
    
    A \emph{trace symbol} of type $d$ is a sum of the form
    \[
      t=\sum_{j=0}^d t_j \partial_+^j, \quad t_j \in S^{m-j}(\R^{n-1}, \R^{n-1}; \Si'(\R_+), \C),
    \]
    where $t$ is in $S^{m}(\R^{n-1}, \R^{n-1}; H^{s_1, s_2}(\R_+), \C)$ 
    and $\partial_+$ is the derivative in the normal direction,
\begin{align*}
 \partial_+= r^+ \partial_{x_n}e^+
 \in S^{1} (\R^{n-1}, \R^{n-1}; H^{\s}\pt{\R_+}, H^{\s-(1,0)}\pt{\R_+}).
 \end{align*}
   \item[iii)] A \emph{singular Green symbol} of order $m$ and type zero is an element of
    \begin{align*}
      &S^m(\R^{n-1}, \R^{n-1}; \Si'(\R), \Si(\R_+))= \\
      &\limproj_\mathbf{s} S^m(\R^{n-1}, \R^{n-1}; H_0^{-\mathbf{s}}
      (\overline{\R}_+),H^\mathbf{s}(\R_+) );
    \end{align*}
    this also is a symbol in $S^{m}(\R^{n-1}, \R^{n-1}; H^{s_1, s_2}(\R_+), \Si(\R_+))$, provided $s_1>-\frac{1}{2}$.
    A \emph{singular Green symbol} of  order $m$ and type $d$ is a sum of the  form
    \[
      g=\sum_{j=0}^d g_j \partial_+^j, \quad g_j \in S^{m-j}(\R^{n-1}, \R^{n-1}; \Si'(\R_+), \Si(\R_+)).
    \]
    Obviously, $g$ then is in $S^m(\R^{n-1}, \R^{n-1}; H^{s_1, s_2}(\R_+), \Si(\R_+))$, $s_1> d- \frac{1}{2}$.
  \end{itemize}

\begin{rem}
  \label{rem:dirac}
  The trace operator $\gamma_j$ is a trace symbol of order $j+\frac{1}{2}$ and type $j+1$,
  see {\rm \cite{SC01}}.
\end{rem}

\noindent
In Definition \ref{def:FIOBdM}, below,  we employ the notation introduced in 
\eqref{eq:omega+}.

\begin{Def}
  \label{def:FIOBdM}
  Let $\chi$ be an admissible symplectomorphism,
  and let $\chi_{\partial}$ be the induced symplectomorphism at the boundary.
  For $m \in \R$ and $d\leq \max\ptg{m, 0}$ we define
  \[
    \mathcal{A}:=\left(
    \begin{array}{cc}
    r^+ A^\chi e^+ + G^{{\chi}_\partial} & K^{\chi_\partial}\\
    T^{\chi_{\partial}}& S^{\chi_\partial}
    \end{array} \right),
  \]
  and we write $\mathcal{A} \in \mathscr{B}^{m,d}_\chi \pt{\Omega_x^+\times\Omega_y^+}$, if
   $A^\chi \in I^m_{comp}(\Omega_x\times\Omega_y, \widetilde{\Lambda})$,
  the symbol of $A^\chi$ satisfies the transmission condition, 
   and $G^{\chi_\partial}$, $K^{\chi_\partial}$, $T^{\chi_\partial}$ and $S^{\chi_\partial}$ are FIOs defined on the boundary,
  with Lagrangian submanifolds defined by $\chi_\partial$. Their respective symbols are
  \begin{enumerate}[\rm (i)]
    \item a singular Green symbol of order $m$ and type $d$;
    \item a potential symbol of order $m$;
    \item a trace symbol of order $m$ and type $d$;
    \item a usual pseudodifferential symbol of order $m$.
  \end{enumerate}
\end{Def}
\begin{rem}
  \label{FIOop}
  In the above definition, we consider FIOs with operator-valued symbol. 
By Remark \ref{rem:triv}, 
$\chi_\partial$ is the lift to the cotangent bundle of a diffeomorphism, locally
  represented as $b: \Omega^\partial_y\to\Omega^\partial_x$. 
  One can extend $b$ to 
   \begin{align}
    \label{eq:estb}
    \tilde{b}\colon \Omega^\partial_y\times[0,1) \to \Omega^\partial_x\times[0,1)\colon(y', y_n)\mapsto (b(y'), y_n).
  \end{align}
Then $K^{\chi_\partial}$, $T^{\chi_\partial}$, $G^{\chi_\partial}$ can be considered as  the pullbacks  of pseudodifferential operators with operator-valued symbol
by the diffeomorphism $\tilde{b}$. 
For example, $G^{\chi_\partial}u= \tilde{b}^* G_1 u = G_2 (\tilde{b}^*u) $, 
  where $G_1$, $G_2$ are usual singular Green operators,  and the equality holds modulo  operators with smooth kernels. 
Similarly for $K^{\chi_\partial}$ and $T^{\chi_\partial}$.
\end{rem}

Here we think of functions with compact support as extended by zero to $\R^n$.
If the symplectomorphism $\chi$ is the identity, and therefore $\Omega_x^+=\Omega_y^+$, the class
$\mathscr{B}^{m,d}_{\chi}(\Omega^+_x\times\Omega^+_x)$ of FIOs
coincides with the class $\mathscr{B}^{m,d}(\Omega^+_x)$ 
of operators of order $m$ and class $d$ in 
the Boutet de Monvel calculus, see \cite{BU71,GR96,RS85,SC01}.
 As a consequence of Theorem \ref{th:contsobolev} and of the Sobolev continuity of FIOs defined through 
 operator-valued symbols, we obtain the following theorem:
\begin{Thm}\label{th:localcont}
  Every $\mathcal{A} \in \mathscr{B}_{\chi}^{m,d}(\Omega^+_x\times\Omega^+_y)$ induces a continuous operator
  \[
    \mathcal{A}: H^{s}_\mathrm{comp}(\Omega_y^+) \oplus H^s_\mathrm{comp}(\Omega_y^\partial) 
    \to H^{s-m}_\mathrm{loc}(\Omega_x^+) \oplus H^{s-m}_\mathrm{loc}(\Omega_x^\partial),
  \]
  provided $s> d- \frac{1}{2}$.
\end{Thm}
Now, we analyze the composition of FIOs of Boutet de Monvel type. Recall that we assume  the involved symbols
to have compact support w.r.t. the space variable.
\begin{Thm}
  \label{th:comp}
  Let $\mathcal{B}\in \mathscr{B}_{\chi'}^{m_\mathcal{B}, d_\mathcal{B}} (\Omega^+_x\times\Omega^+_y)$ and
  $\mathcal{A}\in \mathscr{B}^{m_\mathcal{A}, d_\mathcal{A}}_{\chi}(\Omega^+_y\times\Omega^+_z)$
  be FIOs
  of Boutet de Monvel type associated with the symplectomorphisms $\chi$, $\chi'$. Then $ \mathcal{B} \mathcal{A} $ is
  a FIO of Boutet de Monvel type  of order $m=m_{\mathcal{B}}+m_{\mathcal{A}}$ and type 
  $d=\max\ptg{(m_{\mathcal{A}}+d_{\mathcal{B}}), d_{\mathcal{A}}}$
  defined by the symplectomorphism $\chi' \circ \chi$, that is 
  $\mathcal{B}  \mathcal{A}
  \in \mathscr{B}^{m, d}_{\chi' \circ \chi }(\Omega^+_x\times\Omega^+_z)$.
\end{Thm}
\begin{proof}
  By definition we can write
  \[
    \mathcal{B} \mathcal{A}= \left(\begin{array}{cc}
    r^+ B^{\chi'}e^+ + G_{\mathcal{B}}^{\chi'_\partial} & K_{\mathcal{B}}^{\chi'_\partial}\\
    T_{\mathcal{B}}^{\chi'_\partial}& S_{\mathcal{B}}^{\chi'_\partial}
    \end{array} \right) 
    \circ
    \left( \begin{array}{cc}
    r^+ A^\chi e^+ + G_{\mathcal{A}}^{\chi_\partial} & K_{\mathcal{A}}^{\chi_\partial}\\
    T_{\mathcal{A}}^{\chi_{\partial}}& S_{\mathcal{A}}^{\chi_\partial}
    \end{array} \right).
  \]
  We start with the composition of the elements in the upper left corner. We can write
  \begin{align*}
    & r^+ B^{\chi'} e^+ r^+ A^\chi e^+= r^+ B^{\chi'} A^\chi e^+ + r^+ B^{\chi'} (e^+r^+ -1)A^\chi e^+.
  \end{align*}
  By the general theory of FIOs, $B^{\chi'}  A^\chi$ is a FIO of order $m_{\mathcal{B}}+m_{\mathcal{A}}$ with 
  canonical transformation  $\chi' \circ \chi$.
  We prove next that the operator $r^+ B^{\chi'}(e^+r^+ -1) A^\chi e^+$ is a FIO on the boundary,
  associated with the canonical transformation $\chi'_{\partial} \circ \chi_\partial$ and  with a Green 
  symbol of order $m_{\mathcal{B}}+ m_{\mathcal{A}}$ and type $(m_{\mathcal{A}})_+= \max\{m_{\mathcal{A}}, 0\}$. As both
  the symplectomorphisms are admissible, the boundary is preserved. 
  $B^{\chi'}$ and  $A^\chi$ 
  are operator-valued FIO  at the boundary with 
  symbols $b$ and $a$ satisfying Assumptions \ref{TechnHyp:1} and
  Lagrangian submanifold induced by $\chi'_\partial$ and 
  $\chi_\partial$, respectively. So, we 
   consider the composition
  \[
    r^+ \op^{\psi'}(b) (e^+ r^+-1) \op^\psi(a)e^+,
  \]
  $\psi'$, $\psi$ being general phase functions associated with $\chi'$ and $\chi$, respectively.
  We start by studying the composition in the normal direction.
  We decompose the symbol
  $a= a_d+ a_0$ as in Proposition \ref{prop:div}.
  First, we analyze the differential part
  \begin{equation}
    \label{eq:deriv}
    r^+ \op^{\psi'}_n(b) \pt{e^+ r^+ -1} \op^\psi_n\pt{ \sum_{j=1}^{m_{\mathcal{A}}} a_j \pt{x', x_n,\xi'}  \xi_n^j} e^+ u,
  \end{equation}
  where $a_j(x',x_n, \xi') \in S^{m_{\mathcal{A}} -j}(\R^{n}\times \R^{n-1})$. 
  Since, on $\Si(\R_+)$
  \begin{equation}
    \label{eq:derep1}
    \xi_n \widehat{e^+ u(y_n)}(\xi_n)= -i \widehat{e^+ \partial_{y_n}u}(\xi_n)-i u(0)\widehat{\delta_0},
  \end{equation}
  induction shows
  \begin{equation}
    \label{eq:derep2}
    \xi_n^j \widehat{e^+u(y_n)}= (-i)^{j} \left(\widehat{e^+ \partial_{x_n}^j u}+ \sum_{l=0}^{j-1} u^{(l)}(0)
    \widehat{\delta}_0^{(j-l-1)}\right).
  \end{equation}
  Now
  \begin{eqnarray}
  \nonumber
\lefteqn{\op^\psi_n\pt{ \sum_{j=1}^{m_{\mathcal{A}}} (-i)^ja_j(x', x_n, \xi') 
   \sum_{l=0}^{j-1} u^{(l)}(0) \widehat{\delta}_0^{\pt{j-l-1}}}}\\
   \label{eq:diracgreen}
   &=&\sum_{j=1}^{m_{\mathcal{A}}}(-i)^j a_j(x', x_n, \xi') 
   \sum_{l=0}^{j-1} u^{(l)}(0) \op_n^{\psi}\pt{\xi^{j-l-1}}(\delta_0).
  \end{eqnarray}
   Following the scheme of the proof of Theorem \ref{teodirac}, one obtains 
  \begin{align}
  \label{eq:epdir}
   \pt{e^+ r^+ -1} \op_n^\psi \pt{\xi_n^{j-l-1}}\pt{\delta_0}= \sum_{k=0}^{j-l-1} d_k(x', \xi')\delta_0^{(k)}
   +c_{j,l}(x', x_n, \xi')
  \end{align}
  where $d_k \in S^{j-l-1-k}(\R^{n-1} \times \R^{n-1})$
  and $c_{j,l}\in S^{j-l- \frac{1}{2}} \pt{\R^{n-1}, \R^{n-1};\C, \Si\pt{\R_-}}$. 
  Observing that $\gamma_{l}: u \mapsto u^{(l)}(0)$ is a trace symbol of order $l+\frac{1}{2}$
  and type $l$, in view of \eqref{eq:epdir} and Theorem \ref{teodirac}, we obtain 
  that $r^+ \op^{\psi'}_n(b) \pt{e^+ r^+ -1}$ applied
  to \eqref{eq:diracgreen} is a Green symbol of order $m_{\mathcal{A}}$ and type $\pt{m_{\mathcal{A}}}_+$.
  Next, we analyze 
  \begin{equation}
    \label{eq:deltak}
    \begin{aligned}
    &\hspace{-8mm}r^+\op^{\psi'}_n(b) (e^+ r^+ -1) \sum_{j=1}^{m_{\mathcal{A}}} \op_n^\psi\pt{ 
    a_j(x',x_n,  \xi') \widehat{e^+\partial^j_{y_n} u}(\xi_n) }\\
    &\hspace{8mm}=r^+\op^{\psi'}_n(b) (e^+ r^+ -1)\sum_{j=1}^{m_{\mathcal{A}}} a_j\op^\psi_n(1)e^+  \partial_+^j u.  
  \end{aligned}
  \end{equation}
  Recall that $e^+ r^+-1= - e^- r^-$ on regular distributions. 
  By Theorem \ref{teosmen}  we find that $ r^- a_j\op^\psi_n(1)e^+$ is a symbol
  in $S^{m_{\mathcal{A}}-j}$ $ \pt{\R^{n-1}, \R^{n-1};  \Si' \pt{\R_+}, \Si \pt{\R_-}}$. Moreover, by Theorem \ref{th:conts} 
  we know that $r^+ \op^{\psi'}_n(b) e^- \in S^{m_{\mathcal{B}}}$ $ \pt{\R^{n-1},\R^{n-1}; \Si \pt{\R_-}, \Si \pt{\R_+}}$.
  Hence, the symbol in \eqref{eq:deltak} is a Green symbol of order $m_{\mathcal{B}}+ m_{\mathcal{A}}$  
  and type $(m_{\mathcal{A}})_+$. 
  \noindent In view of  the decomposition $a= a_d+a_0$, we now  have  to consider 
   \begin{align}  \label{eq:bzero}
     &r^+ \op^{\psi'}_n(b) (e^+r^+ -1) \op^\psi_n(a_0) e^+.
   \end{align}
   Theorem \ref{teosmen} implies that $ r^- \op^\psi_n(a_0)e^+ 
  \in S^{m_{\mathcal{A}}}$ $\pt{\R^{n-1}, \R^{n-1}; \Si' \pt{\R_+}, \Si\pt{\R_-}}$. Observing 
  that $r^+ \op^{\psi'}_n(b)e^-$ is an element of
  $S^{m_{\mathcal{B}}}$ $ \pt{\R^{n-1}, \R^{n-1}; \Si\pt{\R_-}, \Si \pt{\R_+}}$, we see that 
  \eqref{eq:bzero} defines a Green symbol of order $m_{\mathcal{B}}+ m_{\mathcal{A}}$ and type zero. 

  The other compositions can be analyzed
  in a similar way. We omit most of the details. They are all FIOs at the boundary with a Lagrangian distribution defined 
  by $\chi'_\partial \circ \chi_\partial$.
  \begin{enumerate}
    \item $\pt{r^+ B^{\chi'}e^+}  G_{\mathcal{A}}^{\chi_{\partial}}$ has a
    singular Green symbol of order $m_{\mathcal{B}}+m_{\mathcal{A}}$ and type $d_{\mathcal{A}}$.
    \item 
    $G^{\chi'_\partial}_{\mathcal{B}} \pt{ r^+ A^\chi e^+}$ has a singular Green symbol of order $m_{\mathcal{B}}+m_{\mathcal{A}}$ and of type
    $(m_{\mathcal{A}}+ d_{\mathcal{B}})_+=\max\{m_{\mathcal{A}}+d_{\mathcal{B}},0\}$.
   \item $G^{\chi'_\partial}_{\mathcal{B}}  G_{\mathcal{A}}^\chi$ has a Green symbol of 
    order $m_{\mathcal{B}}+m_{\mathcal{A}}$ and type $d_{\mathcal{A}}$. 
   \item $\pt{r^+   B^{\chi'}e^+}  K_{\mathcal{A}}^{\chi_\partial}$ has  a potential symbol of order $m_{\mathcal{B}}+m_{\mathcal{A}}$.
   \item $G^{\chi'_\partial}_{\mathcal{B}} K_{\mathcal{A}}^{\chi_\partial}$ has  a potential symbol of order $m_{\mathcal{B}}+m_{\mathcal{A}}$.
   \item $K^{\chi'_\partial}_{\mathcal{B}}  T_{\mathcal{A}}^{\chi_\partial}$ has a Green symbol of
    order $m_{\mathcal{B}}+m_{\mathcal{A}}$ and type $d_{\mathcal{A}}$.
   \item $K^{\chi'_\partial}_{\mathcal{B}}  S_{\mathcal{A}}^{\chi_\partial}$ has a potential symbol 
    of order $m_{\mathcal{B}}+m_{\mathcal{A}}$.
    \item $T^{\chi'_\partial}_{\mathcal{B}} r^+ A^\chi e^+$ has a
    trace symbol of order $m_{\mathcal{B}}+m_{\mathcal{A}}$ and type $(m_{\mathcal{A}}+d_{\mathcal{B}})_+$. 
    \item $T^{\chi'_\partial}_{\mathcal{B}}  G_{\mathcal{A}}^{\chi_\partial} $ has 
    a trace symbol of order $m_{\mathcal{B}}+ m_{\mathcal{A}}$ and type $d_{\mathcal{A}}$. 
    \item $S^{\chi'_\partial}_{\mathcal{B}}  T_{\mathcal{A}}^{\chi_\partial}$ has a trace
    symbol of order $m_{\mathcal{B}}+ m_{\mathcal{A}}$ and type $d_{\mathcal{A}}$.
    \item $T^{\chi'_\partial}_{\mathcal{B}} K_{\mathcal{A}}^{\chi_\partial}$ has  a symbol in $S^{m_{\mathcal{B}}+m_{\mathcal{A}}}(\R^{n-1}\times \R^{n-1})$.
    \item $S^{\chi'_\partial}_{\mathcal{B}}  S_{\mathcal{A}}^{\chi_{\partial}}$ has
    a symbol in $S^{m_{\mathcal{B}}+m_{\mathcal{A}}}(\R^{n-1}\times \R^{n-1})$.
  \end{enumerate}
  The composition in 1) follows from Remark \ref{FIOop} and the composition laws of operator-valued pseudodifferential operators.
  The  compositions in  3), 4), 5), 6), 7), 9), 10), 11), 12)
  can be treated similarly, exploiting the properties of 
  operator-valued symbols.
  The compositions in 2) and 8)
  are slightly more delicate. Let us analyze the composition in 2).  We first suppose $d_{\mathcal{B}}=0$. 
  By Proposition \ref{prop:div}, $a$ can be written in the form $a=a_d+a_0$. 
  Using \eqref{eq:derep1} and \eqref{eq:derep2}, we can write
  \begin{align}
    \nonumber
    &\hspace{-6mm}\pt{r^+ \op^\psi_n(a_d) e^+} u= \\
    \nonumber
    &\hspace{6mm}\sum_{j=1}^{m_{\mathcal{A}}}(-i)^j  r^+ a_{j}(x', x_n, \xi') 
    \pt{\op^\psi_n(1)e^+\partial_{y_n}^j u +\sum_{l=0}^{j-1} \op^\psi_n (1) \pt{\delta_0^{(j-l-1)} }\gamma_l(u) }.
  \end{align}
  By Theorem \ref{teodirac}, Remark \ref{rem:tras} and the properties of trace operators, the
  sum in $j,l$ can be written as
  \[
    r^+ \op^{\psi}_n(a_d) e^+= \sum_{j=1}^{m_{\mathcal{A}}} \tilde{a}_j(x', \xi') \partial_+^j+  
    \sum_{l=0}^{j-1} k_{l}(x', \xi') \gamma_l(u),
  \]
  $\tilde{a}_j \in S^{m_{\mathcal{A}}-j}\pt{\R^{n-1}, \R^{n-1}, \Si'\pt{\R_+}, \Si'\pt{\R_+}}$,
  $k_l \in S^{m_\mathcal{A} -l -\frac{1}{2}}\pt{\R^{n-1}, \R^{n-1}, \C, \Si\pt{\R_+}}$.
  By the properties of trace operators and the definition of Green symbols of type zero, we then obtain that 
  \(
    g_{\mathcal{B}} \pt{r^+ \op^\psi_n(a_d) e^+}
  \)
  is a Green symbol of order $m_{\mathcal{B}}+ m_{\mathcal{A}}$ and type $(m_{\mathcal{A}})_+$.
  To prove the same result for $a_0$, we notice that $r^+ \op^\psi_n(a_0) e^+$
  extends to a symbol in $S^{m_{\mathcal{A}}} \pt{\R^{n-1}, \R^{n-1}; \Si'\pt{\R_+}, \Si'\pt{\R_+}}$, and
  the assertion follows from the definition of Green symbols.
  If $d_{\mathcal{B}}\not=0$, we see that
  \begin{align*}
   & \hspace{-8mm}\partial_+ \int e^{i \psi(x',x_n,\xi',\xi_n)- i \psi_{\partial}(x',\xi')}a(x',x_n,\xi', \xi_n) \hat{u}(\xi_n) 
    \dbar\xi_n\\
    &\hspace{8mm}=r^+\int e^{ i \psi(x',x_n,\xi',\xi_n)- i \psi_{\partial}(x',\xi')} \tilde{a}(x',x_n,\xi', \xi_n) 
    \hat{u}(\xi_n)\dbar\xi_n,
  \end{align*}
  where $\tilde{a}= \partial_{x_n}a+ \pt{i \partial_{x_n}\psi} a $, which implies 
  $\tilde{a}\in S^{m_{\mathcal{A}}+1}(\R^n\times \R^n)$. 
  Iteratively we can reduce to the case $d_{\mathcal{B}}=0$, 
  raising the order from $m_{\mathcal{A}}$ to $m_{\mathcal{A}}+d_{\mathcal{B}}$. 
  To handle the composition 8), we proceed similarly.
 \end{proof}
We will now show a Egorov type theorem for operators in 
$\mathscr{B}^{m, d}_\chi \pt{\Omega_x^+\times\Omega_y^+}$. In analogy with
the usual calculus of FIOs on closed manifolds, we analyze the adjoint of the operator $\mathcal{A}$.
\begin{Thm}
  \label{th:adj}
  The formal adjoint $\mathcal{A}^*$ of an operator 
  $\mathcal{A} \in  \mathscr{B}^{m, 0}(\Omega_x^+\times\Omega_y^+)$, $m\leq0$,
  is a FIO of Boutet de Monvel type, namely 
  $\mathcal{A}^*\in \mathscr{B}_{\chi^{-1}}^{m, 0}\pt{\Omega_y^+\times\Omega_x^+}$. 
  Moreover, 
  \begin{equation}
    \label{eq:adjointb}
    \mathcal{A}^*=\left(
    \begin{array}{cc}
      r^+	\pt{A^{\chi}}^* e^+ + \pt{G^{\chi_\partial}}^* & \pt{T^{\chi_\partial}}^*\\
      \pt{K^{\chi_{\partial}}}^*& \pt{S^{\chi_\partial}}^*
    \end{array} \right),
  \end{equation}
  where $\pt{A^\chi}^*$ is the formal adjoint of $A^\chi$, and
  the Lagrangian submanifold is defined by the graph of $\chi^{-1}$. The operators
  $\pt{ G^{\chi_\partial} }^*$, $\pt{K^{\chi_\partial}}^*$, $\pt{T^{\chi_{\partial}}}^*$, $\pt{S^{\chi_\partial}}^*$
  appearing in \eqref{eq:adjointb} are the adjoints of $G^{\chi_\partial}$, $K^{\chi_\partial}$, $T^{\chi_\partial}$,
  $S^{\chi_\partial}$, respectively, that is, they are FIOs with Lagragian
  submanifold   given by the graph of $\chi_{\partial}^{-1}$.
 \end{Thm}
\begin{proof}
  Since $m\leq 0$, $A^\chi $ is continuous from $L^2(\R^n)$ to itself. Moreover, 
  $e^+:L^2(\R^n_+) \to L^2(\R^n)$ is continuous and its adjoint is $r^+$.
  So we can write
  \begin{align*}
    \pt{r^+ A^\chi e^+ u, v}_{L^2(\R^n_+)}&=\pt{A^\chi e^+u, e^+v}_{L^2(\R^n)}\\
    &=\pt{e^+u,(A^\chi }^* e^+v )_{L^2(\R^n)}=\pt{u, r^+\pt{A^\chi }^* e^+v}_{L^2(\R^n_+)}.
  \end{align*}
  For the other components of $\mathcal{A}^*$, we apply Remark \ref{FIOop} and recall that
  the adjoint of a Green operator of order $m$ and type $0$ is a Green operator of the same order and type, the adjoint of
  a potential operator of order $m$ is a trace operator of order $m$ and type $0$ and the adjoint of a trace operator
  of order $m$ and type $0$ is a potential operator of order $m$, see \cite{SC01}.
\end{proof}
\begin{Def}
  For every $m\in \Z$  we define the operator
  \[
    [\Lambda^m_+]:=
    \left(
    \begin{array}{cc}
      r^+	 \Lambda^m e^+  & 0\\
      0 & \op(\langle \xi'\rangle^{m})
    \end{array} 
    \right)
  \]
  where $r^+ \Lambda^m e^+: H^m(\R^n_+) \to L^2(\R^n_+)$ is an isomorphism in Boutet de Monvel's calculus. The operator $[\Lambda^m_+]$
  is an element of $\mathscr{B}^{m,0}(\overline{\R^n}_+)$ and it is invertible.
\end{Def}
\begin{Thm}
  \label{th:ego}
  Let $\mathcal{A} \in \mathscr{B}^{m,d}_{\chi}(\Omega_x^+\times\Omega_y^+)$, $d\leq m_+=\max\ptg{m, 0}$, $m \in \Z$, be a FIO of Boutet de Monvel type. 
  Then,
  \begin{itemize}
    \item[i)] If $m\leq0$ and $d=0$, $\mathcal{A}  \mathcal{A}^*$ is a an element of $\mathscr{B}^{2m, 0}
		(\Omega_x^+)$;
    \item[ii)] If $m>0$ and $\omega\in C_c^\infty(\Omega_y)$,
       $(\mathcal{A} \omega [\Lambda^{-m}_+]) (\mathcal{A}\omega [\Lambda^{-m}_+])^*$ 
    is an element of
    $\mathscr{B}^{0,0}(\Omega_x^+)$.
  \end{itemize}
\end{Thm}
\begin{proof}
  The proof of part \emph{i)} essentially follows from Theorem \ref{th:adj} and from Egorov's theorem for FIOs
  on closed manifolds.
  The second part follows from the first, noticing that
  $\mathcal{A} \omega [\Lambda^{-m}_+]$  belongs to $\mathscr{B}_{\chi}^{0, 0}
  (\Omega_x^+\times\overline{\R^n}_+)$ by
  Theorem \ref{th:comp}.
\end{proof}
In general, one cannot expect a Egorov type theorem for FIOs of Boutet de Monvel type of all orders and types.
In fact, the formal adjoint $\mathcal{P}^*$ of an operator  
$\mathcal{P}\in\mathscr{B}^{m,d}(\overline{\R^n}_+)$, with
$m>0$ or $d>0$, in general is not even a Boutet de Monvel operator. 
However, by means of Theorems \ref{th:comp} and \ref{th:ego}, it  is possible to
prove that
\begin{Thm}\label{th:egohs}
  If
  $\mathcal{P} \in \mathscr{B}^{m', d'}(\Omega_y^+)$ and
  $\mathcal{A} \in \mathscr{B}^{m,0}_{\chi}(\Omega_x^+\times\Omega_y^+)$, $m\leq0$, then 
  $ \mathcal{A}  \mathcal{P} \mathcal{A}^*$
  belongs to $\mathscr{B}^{m', d'}(\Omega_x^+)$.
\end{Thm}

\section{Principal Symbols}
\label{sec:prin}
We next  define the boundary principal symbol $\sigma_\partial(\mathcal A)$ 
of a FIO of Boutet de Monvel type
\[
  \mathcal{A}:=
  \left(
  \begin{array}{cc}
      r^+	A^\chi e^+ + G^{\chi_\partial} & K^{\chi_\partial}\\
      T^{\chi_{\partial}}& S^{\chi_\partial}
  \end{array} 
  \right) 
   \in \mathscr{B}^{m,d}_\chi(\Omega_x\times\Omega_y).
\]
We have shown that
$\mathcal{A}$ can be seen as an operator-valued FIO defined on the boundary with a symbol 
belonging to $ S^m \pt{\R^{n-1}, \R^{n-1}; \Si \pt{\R_+} \oplus \C,\Si \pt{\R_+} \oplus \C}$. We will 
now switch to {\em classical} operator-valued symbols, see the Appendix, and show that in this setting we can define 
\begin{equation}
  \label{eq:prsplit}
	\sigma_\partial(\mathcal{A})=
  \sigma_\partial\left(
  \begin{array}{cc}
      r^+	A^\chi e^+  & 0
      \\
      0 & 0
  \end{array} 
  \right) 
  +
  \sigma_\partial\left(
  \begin{array}{cc}
      G^{\chi_\partial} & K^{\chi_\partial}\\
      T^{\chi_{\partial}}& S^{\chi_\partial}
  \end{array} 
  \right).
\end{equation}
In view of the fact that $G^{\chi_\partial}={\tilde b}^* G_1$, 
where $\tilde b$ is the diffeomorphism in \eqref{eq:estb} and $G_1$ is a usual Green operator, 
we already have a suitable principal symbol for $G^{\chi_\partial}$, 
namely the pullback of the principal symbol of $G_1$ under $\chi_\partial$.
Similar arguments apply to $K^{\chi_\partial}, T^{\chi_\partial}, S^{\chi_\partial}$.
Hence we can focus on finding a principal boundary symbol $r^+A^\chi_\partial e^+$ for
$r^+ A^\chi e^+$.

The natural candidate for the boundary principal  symbol of $r^+A^\chi e^+ $ is
\begin{multline}
  \label{eq:etan}
  r^+A^\chi_\partial e^+  \pt{x', \eta'}: u \mapsto
  r^+ \int e^{i x_n \partial_{x_n}\psi(x', 0, \eta', \eta_n)- i y_n \eta_n} 
  a_m(x', 0, \eta', \eta_n) e^+\pt{u} dy_n \dbar \eta_n, \\\quad \pt{x', \eta'} 
  \in \Omega^\partial_x\times T^*_{(b^{-1}(x'),0)}{\Omega_y}\setminus0.
\end{multline}
For simplicity this is written for $x'\in\Omega_x^\partial$ and $\eta'\in T^*_{(b^{-1}(x'),0)}\Omega_y\setminus0$.
To have a consistent definition of principal symbol, we have to express $\eta'$ in terms of $(x',\xi')$.
The symbol $a_m$ is a section of the Maslov bundle. 
Since, in our case, this bundle is trivial near the boundary, it is defined as a section 
of the cotangent bundle restricted to the boundary. Notice that the operator-valued symbol 
in \eqref{eq:etan} is homogeneous of degree $m$ in the sense of Definition \ref{opvaluedsymbols}
when $a_m$ is the homogeneous principal part
of $a\in S^m_\mathrm{cl}(\R^n\times\R^n)\cap S^{m}_\tr (\R^{n} \times \R^{n})$.
\begin{Lem}
 The operator-valued function $r^+ A^\chi_\partial e^+$, defined in \eqref{eq:etan},
 belongs to the symbol class $S^m(\R^{n-1}, \R^{n-1}; \Si(\R_+),\Si(\R_+))$.
\end{Lem}
\begin{proof}
 One can divide the operator in two parts: one with compact support in $x_n$
 and one with symbol vanishing to infinite order at $x_n=0$.
 For the part with compact support we repeat the arguments in the proofs of Theorems \ref{teodirac} and \ref{th:conts}.
 Using essentially integration by parts, one proves that also the second operator belongs to
 $S^m(\R^{n-1}, \R^{n-1}; \Si(\R_+),\Si(\R_+))$.
\end{proof}
In view of Proposition \ref{prop:global} and Assumptions \ref{TechnHyp:1} we shall write
in the sequel $A^\chi_n=\op^\psi_n(a)$.
\begin{Thm}
  \label{th:simbprin}
  Let $a$ and $\psi$ satisfy the Assumptions \ref{TechnHyp:1}. Moreover, assume that $a$ is a classical symbol.
  Then, $r^+A_\partial^\chi e^+$, defined in \eqref{eq:etan},
  satisfies
  \[
    r^+A_n^\chi e^+ - r^+ A_\partial^\chi e^+ \in S^{m-1}\pt{\R^{n-1}, \R^{n-1}; \Si\pt{\R_+}, \Si\pt{\R_+}}.
  \]
\end{Thm}
\begin{proof}
  Let $a^0_m(x', \eta)= a_m(x', 0, \eta)$ and $A^\chi_{n,0}= \op^\psi_n(a^0_m)$.
  We claim that
  \begin{equation}
    \label{eq:ridbordo}
    r^+ A_n^\chi e^+ - r^+ A_{n,0}^\chi e^+ \in S^{m-1} \pt{\R^{n-1}, \R^{n-1}; \Si\pt{\R_+}, \Si\pt{\R_+}}.
  \end{equation}
  In the above equation, we consider 
  the phase function $\psi$ extended to $\R^n$ such that the extension coincides, outside a compact subset of 
  $\R^n\times\R^n$, with the standard pseudodifferential phase $x\cdot \eta$, see
  \cite{BCS14} for details. We claim that
  \begin{align}
    \label{eq:opn0}
    r^+ A_{n,0}^\chi e^+ \in S^m(\R^{n-1}, \R^{n-1}; \Si(\R_+), \Si(\R_+)).
  \end{align}
  To see this  we split $r^+ A_{n,0}^\chi e^+$ into two parts. 
  The first one has a symbol vanishing near $\ptg{x_n=0}$ and a
  standard pseudodifferential phase, thus belongs to 
  $S^m(\R^{n-1}, \R^{n-1}; \Si(\R_+), \Si(\R_+))$. The same is true for
  the second  whose symbol is compactly supported w.r.t. the 
  $x_n$-variable, so that we can apply Theorem \ref{th:conts}. 
  Next we note that $a-a_m \in S^{m-1}\pt{\R^{n}\times\R^{n}}$ and write  
  \[
	  a_m(x', x_n, \eta', \eta_n)- a_m(x', 0, \eta', \eta_n)= x_n b(x', x_n, \eta', \eta_n),
  \]
 where $b \in S^{m}(\R^n\times\R^n)$.
 Since the multiplication by $x_n$ is an element in $S^{-1}(\R^{n-1}, \R^{n-1};$ $\Si(\R_+), \Si(\R_+))$,
 it is enough to check that
  \[
    r^+A_{n,0}^\chi e^+ - r^+ A_\partial^\chi e^+ \in S^{m-1}(\R^{n-1}, \R^{n-1}; \Si(\R_+), \Si(\R_+)).
  \]
  Explicitly, we have to consider the seminorms of
  \begin{equation}
    \label{eq:diff}
    \kappa_{\giap{\eta'}^{-1} } r^+\pt{ A_{n,0}^\chi -  A_\partial^\chi } e^+ \kappa_{\giap{\eta'} }.
  \end{equation}
  Evaluating the semigroup actions, we obtain
  \begin{align}
    \nonumber
    &\kappa_{\giap{\eta'}^{-1} } r^+\pt{ A_{n,0}^\chi -  A_\partial^\chi } e^+ \kappa_{\giap{\eta'} }u(x)=\\
    \nonumber
    &r^+ \int e^{i \psi\left(x', \frac{x_n}{\langle \eta'\rangle} , \eta', \eta_n \langle \eta' \rangle \right) - i \psi_{\partial}(x', \eta')} 
    a_m(x', 0, \eta', \eta_n \langle \eta' \rangle)  \widehat{e^+ u}(\eta_n)\dbar \eta_n - \\
    \nonumber
    &r^+ \int e^{i \frac{x_n}{\langle \eta' \rangle} \partial_{x_n}\psi(x', 0, \eta', \eta_n \langle \eta' \rangle )}
    a_m(x', 0, \eta', \eta_n \langle \eta' \rangle) \widehat{e^+ u}(\eta_n) \dbar \eta_n\\
    \nonumber
    =& \int e^{i\frac{x_n}{\langle \eta' \rangle} \partial_{x_n}\psi(x', 0, \eta', \eta_n \langle \eta' \rangle )} 
    \left(   e^{i \psi\left(x', \frac{x_n}{\langle \eta'\rangle} , \eta', \eta_n \langle \eta' \rangle \right) - i \psi_{\partial}(x', \eta') -
    i \frac{x_n}{\langle \eta' \rangle} \partial_{x_n}\psi(x', 0, \eta', \eta_n \langle \eta' \rangle )} -1\right)\\
    \label{eq:risprinc}
    & \hspace{5mm}\cdot a_m(x', 0, \eta', \eta_n \langle \eta' \rangle) \widehat{e^+ u}(\eta_n) \dbar \eta_n.
  \end{align}
  Arguing now as in the proof of Theorem \ref{th:conts}, we can use an excision function and split
  \eqref{eq:risprinc} in 
  \begin{itemize}
   \item[i)] a part which is smoothing w.r.t. the $\eta_n$-variable;
   \item [ii)] one which involves operators  of order $m-j-1$ acting on derivatives of order $j$ of 
   the Dirac distribution;
   \item[iii)] a part with   an arbitrary decay rate w.r.t. the $\eta_n$-variable.
  \end{itemize}
For term \emph{ii)}  we observe that the top order terms of the operators arising from  
$r^+A^\chi_{n,0}e^+$ and $r^+A^\chi_n e^+$,
    applied to $\delta_0^{(j)}$, both agree. This is a consequence of Remark \ref{rem:diractop}, 
    since they both involve the same power of the $y_n$-derivative of 
    the phase function associated with $\chi^{-1}$,
    and the expression is then evaluated at $y_n=0$.

  For the terms \emph{i)}  and \emph{iii)} we notice that the Taylor expansion at $x_n=0$ of
  \begin{align*}
    &e^{i \psi\left(x', \frac{x_n}{\langle \eta'\rangle} , \eta', \eta_n \langle \eta' \rangle \right) - i \psi_{\partial}(x', \eta') -
    i \frac{x_n}{\langle \eta' \rangle} \partial_{x_n}\psi(x', 0, \eta', \eta_n \langle \eta' \rangle )} -1
  \end{align*}
  is simply given by
  \[
    x_n^2 b(x', x_n, \eta', \eta_n\giap{\eta'}),
  \]
  where $b$ is a symbol order one. As the multiplication by $x_n^2$
  is an operator-valued symbol of order $-2$, we gain one order of decay in $\eta'$. To check all the other seminorms, 
  we take advantage of the fact that the symbol can be 
  assumed to have arbitrarily high decay w.r.t.  $\eta_n$.
\end{proof}
The above formulation of the principal symbol is invariant since, with the  notation in
\eqref{eq:compsympl}, \eqref{eq:simplbound}, and the fact that $\chi_\partial$ is the
lift of a diffeomorphism of the boundary,
\[
 \partial_{x_n}\psi(x', 0, \eta', \eta_n)= (\xi_n)^*(y', 0, \eta', \eta_n), 
 \quad  x'=x'^*(y', 0,\eta', \eta_n)=x'^*_\partial(y', \eta')= x'^*_\partial(y').
\]
So we can write

\begin{Def}\label{def:prinbdsym}
 Let $A^\chi$ be a FIO with phase function and symbol satisfying Assumptions \ref{TechnHyp:1}, and let its principal symbol
 be denoted by $\sigma(A)$. Then, its boundary principal symbol is defined, on $T^*\Omega_x^\partial$, as
 \begin{equation}
    \label{eq:xin}
 \begin{aligned} 
    r^+ &A_\partial^\chi  e^+ (x',\xi' ):u\mapsto
    \\
    &r^+ \int e^{i x_n (\xi_n)^*( x', 0, (\xi'^*_\partial)^{-1}(\xi'), \eta_n)- i \eta_n y_n  } 
    \sigma(A)(x', 0, (\xi'^*_\partial)^{-1}(\xi'), \eta_n)
    e^+( u) dy_n \dbar\eta_n.
 \end{aligned}
 \end{equation}
This is the homogeneous principal symbol in the sense of Definition \ref{opvaluedsymbols} in the Appendix.
\end{Def}

\begin{Def}
  Let $\mathcal{A}$, $\chi$, $a$, $G^{\chi_\partial}$, $T^{\chi_\partial}$, $S^{\chi_\partial}$ 
  be as in Definition \ref{def:FIOBdM}. We define, on $T^*\Omega_x^\partial$,
  \begin{multline}
    \sigma_{\partial}(\mathcal{A})(x', \xi')=\left(
    \begin{array}{cc}
      r^+A^\chi_{\partial}  e^+ + \sigma_\partial(G^{\chi_\partial}) 
                & \sigma_\partial(K^{\chi_\partial})\\
      \sigma_\partial(T^{\chi_\partial})& \sigma_\partial(S^{\chi_\partial})
    \end{array} 
    \right) \\
    \in S^{m} \pt{\R^{n-1}, \R^{n-1}; \Si \pt{\R_+} \oplus \C,\Si \pt{\R_+} \oplus \C },
  \end{multline}
  where $\sigma_\partial(G^{\chi_\partial}), 
  \sigma_\partial(K^{\chi_\partial}), \sigma_\partial(T^{\chi_\partial}), 
  \sigma_\partial(S^{\chi_\partial})$ 
  are the homogeneous principal symbols of the corresponding operators.
  The symbol $\sigma_\partial(\mathcal{A})$ is called the boundary principal symbol of $\mathcal{A}$.
\end{Def}
\begin{Prop}
  \label{prop:comp}
  If $\mathcal{A} \in \mathscr{B}_{\chi_1}^{m_1, d_1} \pt{\Omega_x\times\Omega_z}$ and 
  $\mathcal{B} \in \mathscr{B}_{\chi_2}^{m_2, d_2}\pt{\Omega_z\times\Omega_y}$
  then $\sigma_\partial(\mathcal{A \mathcal{B}})= \sigma_\partial (\mathcal{A}) \,\sigma_\partial(\mathcal{B})$.
\end{Prop}
\begin{proof}
  Let us write
  \[
    \mathcal{A}
    =\left(
    \begin{array}{cc}
      r^+A^{\chi_1} e^+ + G_1^{(\chi_1)_\partial} & K^{(\chi_1)_\partial}_1\\
      T^{(\chi_1)_\partial}_1& S^{(\chi_1)_\partial}_1
    \end{array} 
    \right)
  \]
  and
  \[
    \mathcal{B}
    =\left(
    \begin{array}{cc}
      r^+B^{\chi_2}  e^+ + G_2^{(\chi_2)_\partial} & K_2^{(\chi_2)_\partial}\\
      T_2^{(\chi_2)_\partial}& S_2^{(\chi_2)_\partial}
    \end{array} 
    \right).
  \]
  
  It is clear that one has to consider only the FIO part; the Green, potential, trace and pseudodifferential 
  terms in the lower right corner are simple compositions.
  The definition of boundary principal symbols implies that
  \begin{align*}
    &r^+A^{\chi_1}_\partial e^+- r^+A^{\chi_1}_n e^+  \in S^{m_1-1}(\R^{n-1}, \R^{n-1}; \Si(\R_+), \Si(\R_+))\\
    &r^+ B^{\chi_2}_\partial e^+ -r^+B^{\chi_2}_n e^+ \in S^{m_2-1}(\R^{n-1}, \R^{n-1}; \Si(\R_+), \Si(\R_+)).
  \end{align*}
  Therefore, we can write
  \[
    r^+ A^{\chi_1}_n e^+  r^+ B^{\chi_2}_n e^+= r^+A^{\chi_1}_\partial e^+  r^+ B^{\chi_2}_\partial e^+  + r 
  \]
  where $ r \in S^{m_1+m_2-1}(\R^{n-1}, \R^{n-1}; \Si(\R_+), \Si(\R_+))$. 
\end{proof}
Theorem \ref{th:adj} describes precisely $\mathcal{A}^*$ for
$\mathcal{A} \in \mathscr{B}_{\chi}^{0,0}\pt{\Omega_x\times\Omega_y}$. 
One can then easily prove the next proposition. 

\begin{Prop}
  \label{prop:adj}
  For any $\mathcal{A}\in\mathscr{B}^{0,0}_{\chi}(\Omega_x\times\Omega_y)$ we have  
  $\sigma_\partial \pt{\mathcal{A}^*}= \ptq{\sigma_\partial(\mathcal{A})}^*$.
\end{Prop}

\begin{Prop}
  \label{prop:sim}
  Let  $\mathcal{A}$
  be an operator belonging to $\mathscr{B}^{0,0}_{\chi}(\Omega_x\times\Omega_y)$ whose 
  boundary principal symbol $\sigma_{\partial}(\mathcal{A})$ is invertible on $T^*\Omega_x^\partial$.
  Then $[\sigma_\partial(\mathcal{A})]^{-1}$ is a boundary principal symbol of an operator
  belonging to $\mathscr{B}^{0,0}_{\chi^{-1}}(\Omega_y\times\Omega_x)$.
\end{Prop}
\begin{proof}
  Proposition \ref{prop:adj} implies that $\sigma_\partial(\mathcal{A}^*)= [\sigma_\partial(\mathcal{A})]^*$.
  From Theorem \ref{th:ego} we know that $\mathcal{\mathcal{A}}  \mathcal{A}^*=\mathcal{P}$ is a usual Boutet de Monvel  operator.
  Proposition \ref{prop:comp} shows that $\sigma_\partial( \mathcal{A}  \mathcal{A}^*)= \sigma_\partial(\mathcal{A}) \,
   \sigma_\partial(\mathcal{A}^*)$
  is invertible. Hence, by the general theory of Boutet de Monvel's calculus,  
  $\ptq{\sigma_\partial \pt{ \mathcal{A}  \mathcal{A}^*}}^{-1}$
  is the boundary principal symbol of a Boutet de Monvel operator $\mathcal{P}'$. The inverse of $\sigma_\partial(\mathcal{A})$ can now be written
  as 
  \[
    [\sigma_\partial(\mathcal{A})]^{-1} = \sigma_\partial(\mathcal{A^*})  \,\sigma_\partial(\mathcal{P'})= 
    \sigma_\partial(\mathcal{A}^* \mathcal{P'} ).
  \]
  Hence $[\sigma_\partial(\mathcal{A})]^{-1}$ is the boundary principal symbol of $\mathcal{A}^*  \mathcal{P}'\in
  \mathscr{B}^{0,0}_{\chi^{-1}}\pt{\Omega_x\times\Omega_y}$, by Theorem \ref{th:comp}.
\end{proof}
\begin{rem}
  The proof of Proposition \ref{prop:sim} implies that the FIO part in the upper left corner of
  $\ptq{\sigma_\partial \pt{\mathcal{A}}}^{-1}$ is $r^+ B^{\chi^{-1}}_\partial e^+$, where $B^{\chi^{-1}}$ is
  a parametrix of $A^{\chi}$, hence a FIO associated with the symplectomorphism $\chi^{-1}$.
\end{rem}

\begin{Def}
  An operator $\mathcal{A} \in \mathscr{B}^{m,d}_{\chi}(\Omega_x\times\Omega_y)$, $d\leq m_+$,
  is called elliptic, if
  \begin{enumerate}[\rm(i)]
    \item The interior symbol $\sigma(\mathcal{A})$ is elliptic. That is, there exists a FIO 
    $B^{\chi^{-1}}$  of order $-m$  such that $A^\chi B^{\chi^{-1}} - \id$ and
    $B^{\chi^{-1}} A^{\chi} - \id$ are FIOs of order  $-1$ on $\Omega_x\times\Omega_x$ and
    $\Omega_y\times\Omega_y$ respectively.
    \item The boundary principal symbol
    \[
      \sigma_{\partial}(\mathcal{A})(x', \xi'): \Si(\R_+) \oplus \C \to \Si(\R_+) \oplus \C
    \]
    is an invertible operator for any $(x',\xi')\in T^*\Omega_x\setminus 0$.
  \end{enumerate}
\end{Def}

\begin{Thm}
\label{th:parametrixloc}
  If $\mathcal{A} \in \mathscr{B}^{m,d}_{\chi}(\Omega_x\times\Omega_y)$, 
  $d\leq m_+$, is elliptic, then there exists an operator
  $\mathcal{B} \in \mathscr{B}_{\chi^{-1}}^{-m, d_-}(\Omega_y\times\Omega_x)$, 
  $d_-= \max\{-m,0\}$,  such that 
  $\mathcal{A}  \mathcal{B} $ and 
  $\mathcal{B}  \mathcal{A}$ are both equal to the identity modulo lower order operators.  
\end{Thm}
\begin{proof}
  Using an order reduction it is sufficient to consider the case $m=d=0$. 
  Since $\mathcal{A}$ is elliptic, there exists a parametrix
  $B^{\chi^{-1}}$  of $A^\chi$ such that $B^{\chi^{-1}}  A^\chi$ and $A^\chi  B^{\chi^{-1}}$ 
  are equal to the identity modulo lower order operators. Let 
  \[
    [\sigma_{\partial} (\mathcal{A})]^{-1} = 
    \left(
    \begin{array}{cc}
      r^+ B^{\chi^{-1}}_\partial e^+ + g' &  k'\\
      t' & s'
    \end{array}
    \right).
  \]
  We set
  \[
    \mathcal{B}= \left(
     \begin{array}{cc}
	r^+ B^{\chi^{-1}} e^+ + \op^{\chi^{-1}_\partial} (g') &   \op^{\chi^{-1}_\partial}(k')\\
	\op^{\chi^{-1}_\partial} \pt{t'} & \op^{\chi^{-1}_\partial} \pt{s'}
    \end{array}
    \right).
   \]
By construction, the composition $\mathcal{A}  \mathcal{B}$ is equal, modulo lower order terms, to the identity  outside a neighborhood of the boundary.
   
  The definition of the boundary principal symbol and Proposition \ref{prop:sim} imply
  that
  \begin{align}
    \nonumber
    &r^+ A^\chi_n e^+ - r^+ A^\chi_\partial e^+= r_1 \in S^{-1}(\R^{n-1}, \R^{n-1}; \Si(\R_+), \Si(\R_+))\\
    \label{eq:resto}
    & r^+ B^{\chi^{-1}}_ne^+ - r^+ B^{\chi^{-1}}_\partial e^+=r_2 \in  
      S^{-1}(\R^{n-1}, \R^{n-1}; \Si(\R_+), \Si(\R_+)).
  \end{align}
  Since the Maslov bundle is trivial for both $\chi$ and $\chi^{-1}$,
  $\mathcal{A}$ and $\mathcal{B}$ can be considered as operator-valued FIOs at the boundary, with
  a Lagrangian submanifold defined by the graphs of $\chi_\partial$ and $\chi^{-1}_\partial$ respectively. 
  Then we can write
  \[
    r^+ A^\chi_n e^+ r^+B^{\chi^{-1}}_n e^+= (r_1+ r^+A^\chi_\partial e^+)  (r_2+ r^+B^{\chi^{-1}}_\partial e^+).
  \]
  Since $r^+A^\chi_\partial e^+ $ and $ r^+B^{\chi^{-1}}_\partial e^+ $ are elements of 
  $ S^0(\R^{n-1}, \R^{n-1}; \Si(\R_+), \Si(\R_+))$ 
the relations \eqref{eq:resto}
and the composition laws of Theorem \ref{th:comp} imply that $r^+ A^\chi_n e^+$ $r^+B^{\chi^{-1}}_n e^+$
  is equal to the identity up to a  symbol of order $-1$. 
Hence $\mathcal{A}\mathcal{B}=\id$ modulo lower order operators. The argument for $\mathcal{B}\mathcal{A}$ is analogous.
  \end{proof}


\section{FIOs of Boutet de Monvel Type on Manifolds}
\label{sec:manif}
Let $X$ and $Y$ be two compact manifolds with boundary, and $\chi$ be an admissible 
symplectomorphism $\chi\colon T^* Y \setminus 0 \to T^* X \setminus 0$.
We recall that $\chi$ extends to a symplectomorphism 
$\tilde{\chi}\colon T^* \widetilde{Y} \setminus 0 \to T^* \widetilde{X} \setminus 0$
of open neighborhoods of $Y$ and $X$, respectively. We denote by $\widetilde{\Lambda}$ the
graph of $\tilde{\chi}$. As pointed out in Section \ref{sec:trans}, the restriction $\chi_\partial$ of
$\chi$ is the lift of a diffeomorphism $b\colon\partial Y\to\partial X$. As in \eqref{eq:estb}, we can
extend $b$ to a diffeomorphism of collar neighbourhoods $\tilde{b}\colon\partial Y\times[0,1)\to
\partial X\times[0,1)$ by setting $\tilde{b}(y',y_n)=(b(y'),y_n)$.
We will now  define FIOs of Boutet de Monvel type
\begin{equation}
  \label{eq:opglob}
  \mathcal{A}
  =\left(
    \begin{array}{cc}
      r^+ A^\chi e^+ + G^{\chi_{\partial}} & K^{\chi_{\partial}} \\
      T^{\chi_\partial} & S^{\partial}
    \end{array}\right)
    \colon
  \begin{array}{c}
      C^\infty(Y, E_1)\\
      \oplus\\
      C^\infty(\partial Y, F_1)
  \end{array}
  \to 
  \begin{array}{c}
    C^\infty(X, E_2)\\
    \oplus\\
    C^\infty(\partial X, F_2)
  \end{array} ,
\end{equation}
acting between vector bundles over $X$, $Y$ and their boundaries.

\begin{Def}
  \label{def:glob}
  The linear operator $\mathcal{A}$ in  \eqref{eq:opglob} belongs to 
  $\mathscr{B}_\chi^{m,d}(X \times Y)$, if
  \begin{enumerate}[\rm (i)]
  \item
   $A^{\chi} $ is a FIO defined by $\widetilde{\Lambda}$ and its local symbols 
   satisfy the transmission condition;
   \item in a collar neighborhood of the boundary,
  $G^{\chi_{\partial}}$, $K^{\chi_\partial}$, $T^{\chi_\partial}$ and $S^{\chi_\partial}$
  are, modulo operators with smooth kernel, the pullbacks of standard singular Green, potential, trace, and pseudodifferential
  operators by  $\tilde{b}$. 
\end{enumerate}
\end{Def}

The following theorem is obvious from Theorem \ref{th:localcont}.
\begin{Thm}\label{th:Hscont}
Let $\mathcal{A} \in \mathscr{B}^{m,d}_\chi(X \times Y)$, $d\leq m_+$, and $s\in \R$ with 
$s>d-1/2$. Then  $\mathcal A$ extends to a bounded operator
 \begin{equation}
  \label{eq:opglobhs}
  \mathcal{A}
    \colon
  \begin{array}{c}
      H^s(Y, E_1)\\
      \oplus\\
      H^s(\partial Y, F_1)
  \end{array}
  \to 
  \begin{array}{c}
    H^s(X, E_2)\\
    \oplus\\
    H^s(\partial X, F_2).
  \end{array} 
\end{equation}
\end{Thm}

\begin{Def}
  \label{def:prinman}
  Let $\mathcal{A}  \in \mathscr{B}^{m,d}_\chi(X \times Y)$. 
  We define  
  \begin{itemize}
    \item[\rm(i)] the interior principal symbol  $\sigma(\mathcal{A})$, as
    the principal symbol $\sigma(A^\chi)$ of $A^\chi$,
     restricted to the Lagrangian submanifold associated with the graph of $\chi$;
    \item[\rm(ii)] the boundary principal symbol $\sigma_\partial(\mathcal{A})$ 
    \begin{align*}
      &\sigma_{\partial}(\mathcal{A}) (x', \xi')= \left(
      \begin{array}{cc}
	r^+ A ^\chi_\partial(x',\xi') e^+ +
	\sigma_\partial(G^{\chi_\partial})(x', \xi') & \sigma_\partial(K^{\chi_\partial})(x', \xi') \\
	\sigma_\partial(T^{\chi_\partial})(x', \xi') & \sigma_\partial(S^{\chi_\partial})(x', \xi')
      \end{array}\right),
    \end{align*}
    for  $(x', \xi') \in T^*\partial X \setminus 0$. Here $\sigma_\partial(G^{\chi_\partial})$, 
    $\sigma_\partial(K^{\chi_\partial})$, $\sigma_\partial(T^{\chi_\partial})$, $\sigma_\partial(S^{\chi_\partial})$ are the principal symbols of the corresponding operators. We consider $\sigma_{\partial}(\mathcal{A}) (x', \xi')$ as an operator
    \[
    \sigma_{\partial}(\mathcal{A}) (x', \xi')\colon
      \begin{array}{c}
      \Si(\R_+, E_1^*)_{(y',\eta')}\\
      \oplus\\
      F^*_{1,{(y',\eta')}}
  \end{array}
  \to 
  \begin{array}{c}
      \Si(\R_+, E_2^*)_{(x',\xi')}\\
      \oplus\\
      F^*_{2,{(x',\xi')}}
  \end{array}
    \] 
    Here, $\chi_\partial(y',\eta')=(x',\xi')$, $E^*_1$ is the pullback of $E_1|_{\partial Y}$ to $T^*\partial Y$
    with $\R_+$ identified with the normal to $\partial Y$ at $y'$, $F^*_1$ is the pullback of $F_1$ to
    $T^*\partial Y$, and similarly for $E_2$ and $F_2$.
  \end{itemize}
\end{Def}
The boundary principal symbol defined above has an invariant meaning. Indeed, the principal symbol of $A^\chi$ is invariantly defined, and since the Maslov bundle can be trivialized
in a  neighborhood of the boundary, the operator-valued symbol \eqref{eq:etan} is defined on
$T^*\partial X \setminus 0$ as we have noted after \eqref{eq:etan}.
The other
symbols involved also have an invariant meaning since the corresponding operators are
 compositions of the  diffeomorphism $\tilde b$ 
 and usual operator-valued pseudodifferential operators. 

 In view of Theorem \ref{th:comp}, Proposition \ref{prop:comp} and the well known properties
 of composition of FIOs associated
 with a canonical transformation we obtain the following theorem.
  
 \begin{Thm}
 \label{th:compman}
  Let $\mathcal{B}\in \mathscr{B}^{m_{\mathcal{B}},d_{\mathcal{B}}}_{\chi'} (X\times Y)$ and 
  $\mathcal{A}\in \mathscr{B}^{m_{\mathcal{A}},d_{\mathcal{A}}}_{\chi} (Y\times Z)$ then
  $\mathcal{B}\mathcal{A}$ belongs to 
  $\mathscr{B}^{m_{\mathcal{B}}+m_{\mathcal{A}},d}_{\chi' \circ \chi}(X\times Z)$,
  where $d=\max\ptg{(m_\mathcal{A}+d_\mathcal{B}), d_{\mathcal{A}}}$.
  Moreover, 
  \begin{align*}
   \sigma_\partial(\mathcal{B} \mathcal{A}) &= \sigma_\partial(\mathcal{B}) \sigma_\partial(\mathcal{B}).
  \end{align*}
 \end{Thm}
\begin{rem}
In view of Definition \ref{def:prinman}, the principal interior symbol of the composition
$\mathcal{B}\mathcal{A}$ is given by the usual formula, see Hörmander \cite{HO04},
Theorem 25.2.3 and Section 25.3.  
The formula for the boundary symbol of the composition appears simpler. 
Note, however, that the symplectomorphism $\chi_\partial $ 
enters also in  Definition \ref{def:prinbdsym}. 
\end{rem}

 \begin{Thm}
 \label{th:adjman}
  Let $\mathcal{A}\in \mathscr{B}^{m,0}_{\chi} (X\times Y)$ with $m\le0$. Then,
  $\mathcal{A}^*$ belongs to 
  $\mathscr{B}^{m,0}_{\chi^{-1}}(Y\times X)$, and 
  \begin{align*}
   \sigma_\partial( \mathcal{A}^*) &= \sigma_\partial(\mathcal{A})^*.
  \end{align*}
 \end{Thm}

 \begin{Def}
  An operator $\mathcal{A} \in \mathscr{B}^{m,d}_\chi(X \times Y)$, $d\leq m_+$, is elliptic if
  $\sigma\pt{\mathcal{A}}$ and $\sigma_\partial(\mathcal{A})$ are invertible.
  \end{Def}
\begin{Thm}
 Let $\mathcal{A} \in \mathscr{B}^{m,d}_\chi(X \times Y)$, $d\leq m_+$, be elliptic.  
 Then there exists a parametrix
 $\mathcal{B} \in \mathscr{B}^{-m,d_-}_{\chi^{-1}}(Y \times X)$, $d_-=\max\ptg{-m,0}$, 
 such that $\mathcal{A} \mathcal{B}$ and $\mathcal{B}\mathcal{A} $ are equal to the identity modulo operators with smooth integral kernels.
In particular, $\mathcal A$ in \eqref{eq:opglobhs}  is a Fredholm operator.
\end{Thm}
\begin{proof}
Using an order reduction we may assume $m=d=0$. 
By the composition properties of Theorem \ref{th:compman}, the operator 
$\mathcal{P}=\mathcal{A} \mathcal{A}^*$
is an elliptic Boutet de Monvel operator, hence it admits a parametrix
$\mathcal{C}$ such that 
\[
 \mathcal{A} \mathcal{A}^* \mathcal{C}= \mathcal{P}\mathcal{C}=\id +\mathcal K, \quad \mathcal K \mbox{ a smoothing operator}.
\]
Therefore, $\mathcal A^* \mathcal{C}$ is a right parametrix of $\mathcal{A}$. 
In the same way we obtain a left parametrix.
\end{proof}
From Theorem \ref{th:compman} we also obtain the following Egorov type theorem.
\begin{Thm}
\label{th:egomf}
  If
  $\mathcal{P} \in \mathscr{B}^{m', d'}(Y)$ and
  $\mathcal{A} \in \mathscr{B}^{m,0}_{\chi}(X\times Y)$, $m\leq0$, then 
  $ \mathcal{A}  \mathcal{P} \mathcal{A}^*$
  belongs to $\mathscr{B}^{m', d'}(X)$.
\end{Thm}


\section{Index of Admissible Symplectomorphisms}
\label{sec:index}
Following the idea of Weinstein \cite{WE75}, we will now associate a FIO of Boutet de Monvel 
type with an admissible  
symplectomorphism $\chi: T^*Y\setminus 0 \to T^*X \setminus 0$ and a unitary section of the Maslov bundle 
$u$. 
Choose on $Y$ and $X$ the trivial line bundle and on $\partial Y$ and $\partial X$ the zero bundle.
We consider the FIO of Boutet de Monvel type defined as
\begin{equation}
  \label{eq:index}
  \mathcal{U}^\chi=\left(
  \begin{array}{c}
    r^+ U^\chi e^+  
  \end{array}\right): C^{\infty}(Y) \to C^{\infty}(X),
\end{equation}
where $U^\chi$ has principal symbol $u$. 
\begin{Thm}
 \label{thm:finally}
 The operator $\mathcal{U}$, defined in \eqref{eq:index}, 
 is elliptic, that is the interior and the boundary symbol are both invertible.
\end{Thm}
\begin{proof}
  The interior symbol is invertible by construction. 
  So we have to show the invertibility of  the boundary symbol
 \begin{align}
  \nonumber
  \sigma_\partial \pt{\mathcal{U}^\chi}(x', \xi')  : \Si(\R_+) &\to \Si(\R_+)\\ 
  \label{eq:invs}
  u &\mapsto r^+ \int e^{i x_n \partial_{x_n} \psi (x', 0, (\xi'^*)^{-1}(\xi'), \eta_n)} \widehat{e^+ u}(\eta_n)\dbar \eta_n
 \end{align}
 for each $(x', \xi') \in T^*\partial X\setminus 0$.
  As we have noticed in Equation \eqref{eq:contsob}, $\sigma_\partial \pt{\mathcal{U}^\chi}(x', \xi')\in \Ldug+$. 
  In the sequel it will be easier to 
  consider the operator
  \begin{align}
    \nonumber
    \sigma_\partial \pt{\mathcal{U}^\chi}(x', \eta')  \colon  L^2(\R_+) &\to L^2(\R_+)\\ 
    \label{eq:inv}
    u &\mapsto r^+ \int e^{i x_n \partial_{x_n} \psi (x', 0, \eta', \eta_n)} \widehat{e^+ u}(\eta_n)\dbar \eta_n.
  \end{align}
In the following,  we will first show that $ \sigma_\partial \pt{\mathcal{U}^\chi}(x', \eta')$ is
invertible on a weighted $L^2$-space. From this we will infer the invertibility on $L^2(\R_+)$ 
and on $\Si(R_+)$.

  First notice that for each $\lambda \in \R$
  \begin{align*}
    \kappa_\lambda\colon L^2 (\R_+ ) \to L^2(\R_+)\colon u(x)\mapsto \lambda^{\frac{1}{2}} u(\lambda x)
  \end{align*}
  is a unitary.
  Hence, the invertibility of \eqref{eq:inv} is equivalent to that of
  \[
    \kappa_{\lambda^{-1}} \circ \sigma_\partial\pt{\mathcal{U}^\chi}(x', \eta') \circ \kappa_{\lambda}u(x_n)=  
    r^+ \int e^{i x_n \partial_{x_n} \psi (x', 0, \lambda^{-1}\eta' , \eta_n)} \widehat{e^+ u}(\eta_n)\dbar \eta_n.
  \]
  That is, we can always suppose that $|\eta'|$ is as small as necessary.
We next consider the larger space
  $L^2_{w\pm}:=L^2\pt{\R_\pm, w}$, $w(x)= (1+|x|)^{-2}$, which contains the constants and all bounded continuous functions. 
  We set  $L^2_\omega=L^2(\R, \omega)$.
For $0\leq t \leq 1$, let
  \begin{align*}
      A(t)&=A^\psi_n (t)(x', \eta'): u\mapsto
      \int e^{i x_n \partial_{x_n} \psi (x', 0, t \eta', \eta_n)} \widehat{ u}(\eta_n)\dbar \eta_n \\
      &=
      \kappa_{t}  \int e^{i x_n \partial_{x_n} \psi (x', 0,  \eta', \eta_n)} 
      \widehat{ \kappa_{t^{-1}} u}(\eta_n)\dbar \eta_n.
  \end{align*}
  In view of the fact that $\partial_{x_n}\psi(x', 0, 0, \eta_n)$ is homogeneous of degree $1$ in $\eta_n$,
  we have $\partial_{x_n}\psi(x', 0, 0, \eta_n)= c(x') \eta_n$ for a smooth function
  $c$. Moreover, $c(x')>0$, since $\partial_{\eta_n}\partial_{x_n} \psi(x', 0, 0, \eta_n)>0$, as noticed after
  \eqref{eq:simplbound}. We have
  \begin{eqnarray*}
      r^+A(0)e^+u(x_n)&=&r^+A^\psi_n (0)e^+u(x_n) = 
      r^+ \int e^{i x_n \partial_{x_n} \psi (x', 0, 0, \eta_n)} \widehat{e^+ u}(\eta_n)\dbar \eta_n\\
				       &=& r^+ \int e^{i x_n c(x') \eta_n} \widehat{e^+ u}(\eta_n)\dbar \eta_n= 
				       u(c(x')x_n).
  \end{eqnarray*}
  Being a dilation in the variable by a positive factor,  $r^+A(0)e^+$ is clearly invertible.

  Let us now consider on $L^2_{w+}$ the operator family $P(t)$, 
  $t\in [0,1]$,  given by
  \begin{align}
    \label{eq:contt}
    P (t)=r^+ A(t)e^+  \pt{r^+ A(t)e^+}^*=r^+ A(t)e^+ r^+ A(t)^*e^+. 
  \end{align}
  Since $P(0)$ is invertible on $L^2_{w+}$, invertibility of $P(t)$ for small $t$ will follow, 
  if we prove that $P(t)$ is continuous at $t=0$.
 In this case, $r^+A(t)e^+$ will be surjective on $L^2_{w+}$. 
  A similar argument for $ \pt{r^+ A^*(t) e^+} \pt{r^+ A(t) e^+}$ will imply injectivity for small $t$.
Hence, for $t$ close to zero,  $r^+A(t)e^+$ will be invertible on $L^2_{w+}$.

So our task is to prove the continuity of $P(t)$ in $t=0$. Notice that
  \begin{equation}
    \label{eq:contdiv}
    P(t)=r^+ A(t)e^+ r^+ A(t)^*e^+ = r^+ A(t) A(t)^*e^+ + r^+ A(t) e^- r^- A^*(t)r^+. 
  \end{equation}
The operator $A(t) A^*(t)$ can be written explicitly
  \[
    \pt{A(t) A^*(t)} u(x_n)= \int e^{i (x_n-y_n) \partial_{x_n}\psi(x', 0, t \eta', \eta_n)} u(y_n)dy_n \dbar \eta_n.
  \]
As observed after \eqref{eq:simplbound}, 
  $\partial_{x_n}\partial_{\eta_n} \psi(x', 0, t \eta', \eta_n)$ $\geq \delta>0$, hence
  it is possible to invert $\eta_n\mapsto\partial_{x_n}\psi\pt{x', 0, t \eta', \eta_n}$. 
  Letting $\xi_n=\partial_{x_n}\psi(x', 0, t \eta', \eta_n)$,
   $\eta_n= \eta_n(x', t \eta', \xi_n)$,
  we obtain the pseudodifferential operator
  \begin{equation}
    \label{eq:aas}
    \pt{A(t) A^*(t)} u(x_n)= 
    \int e^{i (x_n  - y_n) \xi_n} \partial_{x_n} \partial_{\eta_n} \psi( x', 0, t\eta', \eta_n(x', t \eta', \xi_n))^{-1}
    u(y_n) dy_n \dbar \xi_n.
  \end{equation}
  Notice that $\partial_{\eta_n}\partial_{x_n}\psi( x', 0, 0, \eta_n(x', 0, \xi_n))^{-1} = c(x')^{-1}$. Therefore
  \[
     A(0) A^*(0)u(x_n)=\int e^{i (x_n-y_n)\xi_n} c(x')^{-1} u(y_n)dy_n \dbar \xi_n= c(x')^{-1} u(x_n),
  \]
 so that $A(0)A^*(0)$ is invertible.
  For each fixed $t$, the operator in \eqref{eq:aas} is an $SG$-pseudodifferential operator of order $(0,0)$,  
  hence continuous on   $L^2_w $. 

We next study the limit as $t\to 0^+$. 
In view of the fact that $\partial_{x_n}\partial_{\eta_n}\psi$ is in general not continuously 
extendable to $\eta=0$, we choose a zero excision function $\zeta$  with 
$\zeta(\xi_n)=0$ for $|\xi_n|<1$ and $\zeta(\xi_n)=1$ for $|\xi|>2$ and  consider the two operators
  \begin{align*}
    \pt{P_1(t) u} (x_n)&= \int e^{i (x_n -y_n ) \xi_n} \partial_{x_n} \partial_{\eta_n} \psi( x', 0, t\eta', \eta_n(x', t \eta',
    \xi_n))^{-1} \zeta(\xi_n) u(y_n) dy_n \dbar \xi_n,\\
    \pt{P_2(t) u }(x_n)&=\int e^{i (x_n -y_n ) \xi_n} \partial_{x_n} \partial_{\eta_n} \psi( x', 0, t\eta', \eta_n(x', t \eta', \xi_n))^{-1}  (1 - \zeta(\xi_n))
    u(y_n) dy_n \dbar \xi_n .
  \end{align*}
  Then $P_1(t)$ is an $SG$-pseudodifferential operator of order $(0,0)$ 
  whose symbol has uniformly bounded seminorms w.r.t. to $t$. Furthermore, 
  \[
    \lim_{t \to 0+}\zeta(\xi_n) \pt{\partial_{x_n}\partial_{\eta_n}\psi(x', 0, t \eta', \eta_n(x', t \eta',\xi_n))^{-1}
    -  c(x')^{-1}}=0
  \]
  uniformly in the topology of $SG$-symbols on $\R$ and therefore
  \[
    \lim_{t \to 0} P_1(t)= P_1(0): u \mapsto \int e^{i (x_n- y_n)\xi_n} c(x')^{-1} \zeta(\xi_n) u(y_n) dy_n \dbar \xi_n
  \]
  in the norm topology of $\Ldwu\empty$.
  
  The kernel of the operator $P_2(t)$ is
  \[
    K_{P_2(t)} \pt{x_n, y_n}= \int e^{i (x_n- y_n)\xi_n} \partial_{x_n} \partial_{\eta_n} 
    \psi( x', 0, t\eta', \eta_n(x', t \eta', \xi_n))^{-1} (1 - \zeta(\xi_n)) 
    \dbar \xi_n.
  \]
  The function $\partial_{x_n} \partial_{\eta_n} \psi( x', 0, t\eta',
  \eta_n(x', t \eta', \xi_n))^{-1} $ is positively homogeneous
  of degree zero in $\pt{\eta', \eta_n}$, hence bounded. 
  As $t\mapsto 0^+$, Lebesgue's theorem of dominated convergence implies that
  \[
    \pt{1- \zeta\pt{\xi_n}} \pt{ \partial_{x_n}\partial_{\eta_n}\psi(x', 0, t \eta', \eta_n(x', t \eta', \xi_n)^{-1}-c(x'))^{-1} }    \to 0 \quad \mbox{in }\; L^1(\R_{\xi_n}),
  \]
  and hence
  \[
    \lim_{t \to 0} K_{P_2(t)}\pt{x_n, y_n}
   = \int e^{i (x_n- y_n)\xi_n}  c(x')^{-1} (1 - \zeta(\xi_n))  \dbar \xi_n
   = K_{P_2(0)}(x_n,y_n), \]
uniformly  on $\R_+\times \R_+$. 

Schur's Lemma then implies that $P_2(t)$ tends to $P_2(0)$ in the norm of
  $\mathscr L( L^2_{w+})$, since
  \begin{align*}
    &\sup_{x_n \in \R_+} \int_{\R_+} \abs{ K_{P_2(t)}(x_n,y_n)- K_{P_2(0)} (x_n,y_n)} \pt{ 1+|y_n| }^{-2} dy_n 
    \\
    &\hspace{20mm}\le\sup_{\R_+\times\R_+} \abs{K_{P_2(t)} (x_n,y_n)- K_{P_2(0)} (x_n,y_n)}\int_{\R_+}\pt{ 1+|y_n| }^{-2} dy_n,\\
    &\sup_{y_n \in \R_+} \int_{\R_+} \abs{ K_{P_2(t)} (x_n,y_n)- K_{P_2(0)}(x_n,y_n) } \pt{ 1+|x_n| }^{-2} dx_n 
    \\
    &\hspace{20mm}\le\sup_{\R_+\times\R_+} \abs{K_{P_2(t)} (x_n,y_n)- K_{P_2(0)}(x_n,y_n) }\int_{\R_+}\pt{ 1+|x_n| }^{-2} dx_n.
  \end{align*}
This implies the continuity of $t \mapsto A(t) A^*(t)$ at $t=0$ in $\mathscr{L}(L^2_\omega)$. 
   In view of the fact that extension by zero and restriction define bounded maps, we also see that     $t \mapsto r^+A(t) A^*(t)e^+$  is continuous at $t=0$ in $L^2_{w+}$.
  Hence, we have proven the continuity of the first operator in the right hand side of \eqref{eq:contdiv}. 
  
  Next consider $G(t)= r^+ A(t) e^- r^- A^*(t)r^+$.
In view of the fact that $r^+A(0)e^-=0$, we want to prove that $G(t)\to 0$ in $L^2_{w+}$. Notice that
  \begin{eqnarray}
    \nonumber
    \lefteqn{\norm{r^+ A(t) e^- r^- A^*(t)e^+ - r^+ A(0) e^- r^- A^*(0)e^+}_{\mathscr L(L^2_{w+}) }} \\
    \nonumber
    &= &\norm{r^+ A(t) e^-   \pt{r^- A^*(t)e^+ - r^- A^*(0)e^+}  + \pt{r^+ A(t) e^-  -r^+ A(0) e^-}
      r^- A^*(0)e^+ }_{\mathscr L(L^2_{w+})} \\
    \nonumber
    &\leq &
    \max \ptg {\sup_t\norm{r^+ A(t)e^-}_{\mathscr L(L^2_{w-},L^2_{w+})} , \norm{r^- A^*(0) e^+}_{\mathscr L(L^2_{w+},L^2_{w-})}}\\
    \label{eq:green1}
    && \cdot\pt{ \norm{ r^- (A^*(t)- A^*(0))e^+ }_{ 
    \mathscr L(L^2_{w+},L^2_{w-})} +
    \norm{r^+ (A(t)-A(0)) e^-}_{\mathscr L(L^2_{w-},L^2_{w+})} }. 
  \end{eqnarray}
  We start by considering $\norm{r^+(A(t) - A(0)) e^-}_{\mathscr L(L^2_{w-},L^2_{w+})}$, 
  that is, the operator
  \[
    u \mapsto r^+ \iint_{\R_+} \pt{e^{i x_n \partial_{x_n}\psi (x', 0, t \eta', \eta_n) - i y_n \eta_n} - e^{i x_n c(x') \eta_n- i y_n \eta_n }} e^- u(y_n)dy_n \dbar \eta_n. 
  \]
The associated kernel is
  \[
    k(t;x_n, y_n)= \int \pt{e^{i x_n \partial_{x_n}\psi (x', 0, t \eta', \eta_n) - i y_n \eta_n} - e^{i x_n c(x') \eta_n - i y_n \eta_n}}\dbar \eta_n, \quad x_n>0, y_n<0.
  \]
  For convenience, we invert the sign of the $y_n$-variable and consider
  \begin{align*}
    \tilde{k}(t;x_n, y_n)&= \int \pt{e^{i x_n \partial_{x_n}\psi(x', 0, t \eta', \eta_n) + i y_n \eta_n} - e^{i x_n c(x')\eta_n  + i y_n \eta_n} } \dbar \eta_n \\
    &= \int e^{i x_n c(x') \eta_n + i y_n \eta_n} \pt{ e^{i x_n \pt{\partial_{x_n}\psi(x', 0, t \eta', \eta_n) -  c(x') \eta_n}}-1}  \dbar \eta_n, \quad x_n>0, y_n>0.
  \end{align*}
Choose a zero excision function $\zeta$ as above
and let
  \begin{align}
    \label{eq:comp}
    &k_1(t;x_n, y_n)=  \int e^{i x_n c(x') \eta_n + i y_n \eta_n} \pt{ e^{i x_n \pt{\partial_{x_n}\psi(x', 0, t \eta', \eta_n) -  c(x') \eta_n}}-1} \zeta(\eta_n) \dbar \eta_n\\
    &k_2(t;x_n, y_n)= \int e^{i x_n c(x') \eta_n + i y_n \eta_n} \pt{ e^{i x_n \pt{\partial_{x_n}\psi(x', 0, t \eta', \eta_n) -  c(x') \eta_n}}-1}(1- \zeta(\eta_n)) \dbar \eta_n.
  \end{align}
  First we analyze $k_1$. For an application of  Schur's Lemma, on $L^2_{w+}$
  it will be enough to check that
\begin{eqnarray}
\sup_{x_n,y_n} |k_1(t;x_n, y_n)|\to 0\ \text{ as }\ t\to 0^+.\label{eq:suffSchur}
\end{eqnarray}
    We recall that the phase function $\partial_{x_n}\psi(x', 0, t\eta', \eta_n)$ satisfies the transmission condition, 
  so that
  \begin{align}
    \nonumber
    &\partial_{x_n}\psi (x', 0, t \eta', \eta_n) -  c(x') \eta_n= \sum_{|\alpha|= 1}\partial_{x_n}
    \partial_{\eta'}^{\alpha} \psi(x', 0, 0, \eta_n) (t \eta')^\alpha +r_1(x', t\eta', \eta_n)\\
    \label{cost}
    &=\sum_{|\alpha|= 1}\partial_{x_n}
    \partial_{\eta'}^{\alpha}\psi(x', 0, 0, 1) (t \eta')^\alpha  +
    \sum_{|\alpha|=2} (t\eta')^\alpha r_{1, \alpha} (x', \sigma(t)\eta', \eta_n), \quad \sigma(t)\in (0, t).
  \end{align} 
  Using integration by parts, we can write
  \begin{eqnarray}
\lefteqn{|k_1(t;x_n, y_n)|=
    \abs{  \int e^{i x_n c(x') \eta_n + i y_n \eta_n}  \pt{ e^{i x_n \pt{\partial_{x_n}\psi(x', 0, t \eta', \eta_n) -  c(x') \eta_n}}-1} 
    \zeta(\eta_n) \dbar \eta_n}\nonumber}   \\
    &\leq& \abs{\int e^{i x_n c(x') \eta_n + i y_n \eta_n} \frac{x_n \partial_{\eta_n}r_{1}\pt{x', t\eta', \eta_n}}{(x_n c(x')+ y_n)}
    e^{i x_n \pt{\partial_{x_n}\psi(x', 0, t \eta', \eta_n) -  c(x') \eta_n}} 
    \zeta(\eta_n)\dbar \eta_n}\nonumber\\
    \nonumber
    &&+ \abs{\int e^{i x_n c(x') \eta_n + i y_n \eta_n}\frac{1}{(x_n c(x')+ y_n)} \pt{ e^{i x_n \pt{\partial_{x_n}\psi(x', 0, t \eta', \eta_n) -  c(x') \eta_n}}-1} 
    \partial_{\eta_n}\zeta(\eta_n) \dbar \eta_n  }\\
    \nonumber
    &\leq& \frac{x_n}{x_n c(x')+ y_n} \int \abs{\partial_{\eta_n} r_1(x', t\eta', \eta_n) } \zeta(\eta_n)\dbar \eta_n\\
    \label{eq:intx}
    &&+\frac{1}{x_n c(x')+ y_n}\int \abs{\pt{ e^{i x_n \pt{\partial_{x_n}\psi(x', 0, t \eta', \eta_n) -  c(x') \eta_n}}-1}} |\partial_{\eta_n}\zeta(\eta_n)| \dbar \eta_n .
\end{eqnarray}
  Notice that $r_{1} (x', t\eta', \eta_n)$ is in general not continuous in $(x', 0,0)$. Nevertheless,
  as a function in $\eta_n$, it satisfies estimates as a symbol in $S^{-1}(\R\times \R)$ for
  $|\eta_n|>1$. Hence, $\zeta(\eta_n) \partial_{\eta_n}r_{1}(x', t \eta, \eta_n)$ is integrable
 and uniformly
  bounded for $t\in(0,1]$. Moreover, 
  \[
    \lim_{t\to 0} \zeta(\eta_n) (t\eta')^\alpha \partial_{\eta_n}r_{1, \alpha}(x', \sigma(t) \eta, \eta_n)=0, \quad |\alpha|=2.
  \]
  Notice also that 
  \[
    \frac{x_n}{c(x')x_n+y_n}\leq c(x')^{-1}, 
  \]
 is uniformly bounded for $x_n>0, y_n>0$.  In \eqref{eq:intx},  
 $\partial_{\eta_n}\zeta(\eta_n)$ is a function with compact support and, for $\eta\not=0$,   
  %
  \begin{eqnarray*}
\lefteqn{\left|e^{i x_n \pt{\partial_{x_n}\psi(x', 0, t \eta', \eta_n) -  c(x') \eta_n}}-1\right|
	 \le
	 \left|x_n \pt{\partial_{x_n}\psi(x', 0, t \eta', \eta_n) -  c(x') \eta_n}\right|}
	 \\
	 &\le&
	 x_nt|\eta'|\left|\sum_{|\alpha'|=1}
	 \int_0^1\partial_{x_n}\partial^{\alpha'}_{\xi'}\psi(x',0,st\eta',\eta_n)\,ds\right|
	 \le M x_n t |\eta'|,
  \end{eqnarray*}
  since $\partial_{x_n}\partial^{\alpha'}_{\xi'}\psi(x',0,st\eta',\eta_n)$ is uniformly bounded for $|\alpha'|=1$,
  $\eta_n\not=0$, $(x',\eta')\in T^*\partial X\setminus0$, $t\in(0,1]$. 
Lebesgue's theorem on dominated convergence therefore implies that
\eqref{eq:suffSchur} holds and  thus the operator associated with $k_1$ tends to zero 
as $t\to 0^+$ by Schur's lemma. A similar, but simpler, argument holds for $k_2(t;x_n,y_n)$. 
 So we have proven that $r^+A(t)e^- \to 0$ in the norm topology. 
 
 A similar argument shows  $r^- A^*(t) e^+\to 0$. 
 Hence, it is possible to find a constant $C>0$ such that for small $\overline t>0$
  \begin{align*}
    & \norm{ r^+ A(t) e^- }_{\mathscr{L}\pt{L^2_{\omega_-},L^2_{\omega_+} } } \leq C,  \quad
    \norm{ r^- A^*(t) e^+ }_{\mathscr{L}\pt{L^2_{\omega_+},L^2_{\omega_-} } }\leq C, \quad 
    t \in [0, \bar{t}].
    \label{eq:const}
  \end{align*} 
  In view of the above argument and Equation \eqref{eq:green1}, we find
  \begin{align*}
   \lim_{t \to 0} \norm{G(t)}_{\mathscr{L}\pt{L^2_{\omega_+} } }&
   \leq C \pt{\lim_{t \to 0} \norm{r^- A(t) e^+}_{{\mathscr{L}\pt{L^2_{\omega_+},L^2_{\omega_-} } }}+ 
   \norm{r^+ A(t)e^-}_{{\mathscr{L}\pt{L^2_{\omega_-},L^2_{\omega_+} } }}}\\
   &=0.
  \end{align*}

  Therefore, we have proven that
  \[
    \lim_{t\to 0} r^+A(t)e^+  \pt{r^+ A(t)e^+}^*= r^+A(0) e^+  r^+ A^*(0)e^+= c(x')^{-1}\textnormal{Id}
  \]
  in the norm topology of $\mathscr L(L^2_{w+})$.
  It follows that $r^+A(t)e^+ \pt{r^+A(t)e^+ }^*$ is invertible for, say, $t<t_1$. 
This implies that $r^+A(t)e^+$ is surjective on $L^2_{w+}$.
We can apply  the same argument to $t\mapsto\pt{r^+ A(t) e^+}^*  r^+ A(t) e^+$ and 
find that it is also continuous at $t=0$ and hence invertible for small $t$, say $t<t_2$.
In particular, $ r^+ A(t) e^+$ is injective on $L^2_{w+}$.   
Hence $ r^+ A(t) e^+$ in invertible on $L^2_{w+}$ for $t<\min\{t_1,t_2\}$, and the same 
is true for  $\pt{r^+A(t)e^+ }^*$.

In particular, both $r^+A(t)e^+ $ and its $L^2(\R_+)$-adjoint $\pt{r^+A(t)e^+ }^*$  
map $L^2(\R_+)$ to itself and are injective. Therefore, $r^+A(t)e^+ $ is invertible on
$L^2(\R_+)$ for small $t$. 

Finally, consider $r^+A(t)e^+ $ for small $t$ on $\mathscr S(\R_+)$. It clearly is injective, 
since it is so on $L^2(\R_+)$. 
The inverse is $\pt {(r^+A(t)e^+)^*r^+A(t)e^+ }^{-1}(r^+A(t)e^+)^*$. 
As $(r^+A(t)e^+)^*r^+A(t)e^+$ is an operator of order and type zero in Boutet de Monvel's
calculus, so is its inverse by spectral invariance, cf.\ \cite{S99}. In particular, the inverse to $r^+A(t)e^+$ maps 
$\mathscr S(\R_+)$ to itself. Hence $r^+A(t)e^+$ also is surjective on $\mathscr S(\R_+)$
and therefore invertible. 
\end{proof}

In view of Theorem \ref{thm:finally}, we can define the index of an admissible symplectomorphism.
\begin{Def}
  Let $X$ and $Y$ be compact manifolds with boundary, and $\chi: T^*Y \setminus 0 \to T^* X\setminus 0$ an admissible symplectomorphism. Then,
  for each unitary section $u$ of the Maslov bundle associated with $\chi$, we define
  \[
    \textnormal{ind}(\chi, u)= \textnormal{ind} \mathcal{U^{\chi}}
  \]
  with $\mathcal{U^{\chi}}$ as in \eqref{eq:index}.
\end{Def}
As in the case of closed manifolds, the index $\textnormal{ind} (\chi, u)$ 
in general might depend on the chosen unitary section $u$. 
If the Maslov bundle is trivial, or the dimension of $X$ and $Y$ is at least three, 
then it turns out that $\textnormal{ind} (\chi, u)$ is independent of $u$. The argument is standard, and we shortly recall it for the sake of completeness. Consider two unitary sections $u_1$, $u_2$ and the corresponding operators $\mathcal{U}^\chi_{u_1}$ and $\mathcal{U}^\chi_{u_2}$. 
Theorem \ref{th:ego} implies that $\mathcal{U}^\chi_{u_1}(\mathcal{U}^\chi_{u_2})^*$ is an element of $\mathscr{B}^{0,0}(X)$, whose
symbol, modulo lower order terms, coincides with $u_1\overline{u_2}$. If the Maslov bundle is trivial, or the dimension of $X$ and $Y$ is at least three, so that the cospheres are simply connected,  there is a homotopy deforming $u_1\overline {u_2}$ to 1 and thus 
$\mathcal{U}^\chi_{u_1}(\mathcal{U}^\chi_{u_2})^*$ to the identity. Then,
\[
	0=\mathrm{ind}[\mathcal{U}^\chi_{u_1}(\mathcal{U}^\chi_{u_2})^*]
	=\mathrm{ind}\mathcal{U}^\chi_{u_1}+\mathrm{ind}(\mathcal{U}^\chi_{u_2})^*
	=\mathrm{ind}\mathcal{U}^\chi_{u_1}-\mathrm{ind}\mathcal{U}^\chi_{u_2},
\]
as claimed. In this case, we can actually define
\[
	\mathrm{ind}\chi:=\mathrm{ind}(\chi,u).
\]
\section{Appendix: Operator-Valued Symbols}
\label{sec:append}
{\def\skp#1{\langle#1\rangle}
We denote by $H^s(\R^n)$, $s\in\R$, the usual Sobolev space on $\R^n$ and by $H^{s_1,s_2}(\R^n)$ 
the weighted space $\skp{x}^{-s_2}H^{s_1}(\R^n)$. We often write $\mathbf s =(s_1,s_2)$. 
Moreover, we let $H^{\mathbf s}(\R^n_+) =\{u|_{\R^n_+}: u\in H^{\mathbf s}(\R^n)\}$ and 
$H_0^{\mathbf s}(\overline\R^n_+) =
\{u\in H^{\mathbf s}(\R^n):\supp u \subseteq \overline \R^n_+\}$.

Next we recall a few facts on operator-valued symbols. For details see \cite{SC01} or 
\cite{SC91}. Let $E$ and $F$ be Banach spaces with strongly continuous group actions
$\kappa^E$ and $\kappa^F$ of $\R_+$, that is, $\kappa^E:\mathbb R_+\to \mathscr L(E)$ 
is strongly continuous and 
$\kappa^E(\lambda_1\lambda_2) = \kappa^E(\lambda_1)\kappa^E(\lambda_2)$
for $\lambda_1, \lambda_2>0$. 
We will write  $\kappa^E_\lambda$ instead of $\kappa^E(\lambda)$ for $\lambda>0$.
The corresponding notation is used for $\kappa^F$. 

On the spaces $H^{\mathbf s}(\R^n), H^{\mathbf s}(\R^n_+)$ and $H^{\mathbf s}_0(\overline \R^n_+)$, $\mathbf s\in \R^2$, we always use the group action induced by the 
unitary action on $L^2(\R^n)$ given by
$$(\kappa_\lambda u)(x) =\lambda^{-n/2} u(\lambda x), \quad x\in\R^n.$$

\begin{Def}\label{opvaluedsymbols}{
A smooth family $a(y,\eta)$, $y,\eta\in \R^q$, of operators in $\mathscr L(E,F)$ is called an operator-valued symbol of order $m\in \R$, if for all multi-indices $\alpha, \beta$,
$$\|\kappa^F_{\skp{\eta}^{-1}}D^\alpha_\eta D^\beta_y a(y,\eta)\kappa^F_{\skp{\eta}}\|_{\mathscr L(E,F)}
=O(\skp\eta^{m-|\alpha|}).$$
We denote the space of all these symbols by $S^m(\mathbb{R}^q,\mathbb{R}^q;E,F)$.

The symbol $a$ is called homogeneous of degree $m$, provided that
\begin{eqnarray*}\label{homogeneous}
\kappa^F_{\lambda^{-1}}a(y,\lambda\eta) \kappa^E_\lambda =\lambda^m a(y,\eta),
\quad \lambda>0, \eta\not=0,
\end{eqnarray*}
and classical, if it has an asymptotic expansion into homogeneous terms.

%
}\end{Def}

This concept includes the usual pseudodifferential operators, by choosing
$E=F=\C$ with the trivial group action.

The definition extends to more general spaces $E$ and $F$. 
Noting that 
$\mathscr S(\R_+) = \proj_{\mathbf s\in \R^2} H^{\mathbf s}(\R_+)$, and 
$\mathscr S'(\R_+) = \ind _{\mathbf s\in\R^2} H_0^{\mathbf s}(\overline\R_+)$
we let
\begin{eqnarray*}
S^m(\mathbb{R}^q,\mathbb{R}^q;\C,\mathscr S(\R_+))
&=& \proj_{\mathbf s\in\R^2} S^m(\mathbb{R}^q,\mathbb{R}^q;\C,H^{\mathbf s}(\R_+))\\
S^m(\mathbb{R}^q,\mathbb{R}^q;\mathscr S'(\R_+),\C)
&=& \proj_{\mathbf s\in\R^2} S^m(\mathbb{R}^q,\mathbb{R}^q;H^{\mathbf s}_0(\overline\R_+),\C)\\
S^m(\mathbb{R}^q,\mathbb{R}^q;\mathscr S'(\R_+),\mathscr S(\R_+))
&=& \proj_{\mathbf s\in\R^2} S^m(\mathbb{R}^q,\mathbb{R}^q;H^{\mathbf{-s}}_0(\overline\R_+),H^{\mathbf s}(\R_+))\\
S^m(\mathbb{R}^q,\mathbb{R}^q;\mathscr S(\R_+),\mathscr S(\R_+))
&=& \ind_{\mathbf t\in\R^2}\proj_{\mathbf s\in\R^2} 
S^m(\mathbb{R}^q,\mathbb{R}^q;H^{\mathbf{-t}}_0(\overline\R_+),H^{\mathbf s}(\R_+)).
\end{eqnarray*}

\begin{Lem}
Let $a\in S^m(\R^n\times \R^n)$. For fixed $(x',\xi')\in \R^{n-1}\times \R^{n-1}$ 
$$(x_n,\xi_n)\mapsto  a(x',x_n,\xi',\xi_n)$$
defines a symbol in $ S^m(\R\times \R)$. For $u\in \mathscr S(\R)$ we let 
$$\left(\op_n(a)(x',\xi')\right)u(x_n)  =\int e^{ix_n\xi_n}a(x',x_n,\xi',\xi_n)\hat u(\xi_n)\dbar\xi_n. 
$$
Then $\op_n(a)$ is an operator-valued symbol in 
$S^m(\R^{n-1}, \R^{n-1},H^{s_1,s_2}(\R_+), H^{s_1-m,s_2}(\R_+))$ 
for all $s_1,s_2\in \R$. 
\end{Lem}

As
$\kappa_{\skp{\xi'}}^{-1} \op_n(a) \kappa_{\skp{\xi'}}$ has the symbol 
$a(x',x_n/\skp{\xi'},\xi',\xi_n\skp{\xi'})$, the assertion follows from the usual symbol estimates. 
The same consideration shows

\begin{Lem}
For $s_1,s_2\in\R$ multiplication by $x_n$ defines an operator-valued symbol in 
$S^{-1}(\R^{n-1},\R^{n-1},H^{s_1,s_2}(\R), H^{s_1,s_2+1}(\R))$. \end{Lem}


With  $a\in S^m(\R^q,\R^q;E, F)$ we associate the pseudodifferential operator $\op(a)$,
defined on the space $\mathscr S(\R^q,E)$ of rapidly decreasing $E$-valued functions by 
$$\op(a) u(y) =\int e^{iy\eta}a(y,\eta)\hat u(\eta)\, \dbar\eta.$$
It maps $\mathscr S(\R^q,E)$ continuously to $\mathscr S(\R^q,F)$. Here $\hat u$ is the vector-valued Fourier transform
$$\hat u(\eta) = \int e^{iy\eta} u(y) \, \dbar y.$$

\begin{Thm}Let $a_1\in S^{m_1}(\R^q,\R^q;F,G) $ and $a_2\in S^{m_2}(\R^q,\R^q;E,F)$
for Banach spaces $E$, $F$ and $G$. 
Then $\op(a_1)\op(a_2)=\op(a)$ for an operator-valued symbol $a\in S^{m_1+m_2}(\R^q,\R^q;E,G)$ with the asymptotic expansion 
$$a(y,\eta) \sim \sum_{\alpha}\frac1{\alpha!} \partial_\eta^\alpha a_1(y,\eta) D^\alpha_y a_2(y,\eta).$$ 
\end{Thm}
The proof is as in the standard case. 
 
\begin{Def}\label{wedge}
Let $s\in \R$ and $E$ be a Banach space.  
The wedge Sobolev  space $\mathscr W^s(\R^q,E)$ is the completion of $C_c^\infty(\R^q,E)$
with respect to the norm
$$
\|u\|^2_{\W^s(\R^q,E)} 
= \int \skp\eta^{2s}\|\kappa_{\skp\eta}^{-1}\hat u(\eta)\|^2_E\, \dbar\eta . 
$$
\end{Def}

The definition extends to projective limits as above. 
A direct computation shows the first identity in the following lemma; the second then is immediate.
\begin{Lem} 
 $\mathscr W^s(\R^{n-1}, H^s(\R))= H^s(\R^n)$ and 
$\mathscr W^s(\R^{n-1}, H^s(\R_+))= H^s(\R^n_+)$, $s\in \R$.
 \end{Lem}

A proof of the following statement can be found in \cite{Seiler99}.

\begin{Thm}Let $a\in S^m(\R^q,\R^q; E, F)$ for Hilbert spaces $E$ and $F$. Then 
$$\op(a) : \mathscr W^s(\R^q,E) \to \mathscr W^{s-m}(\R^q,F)$$
is bounded, $s\in \R$. 
\end{Thm}

As $\mathscr W^0(\R^{n-1},L^2(\R_+))=L^2(\R^n_+)$, we can also consider the formal
adjoint $\op(a)^*$ of $\op(a)$ with respect to the $L^2(\R^n)$ scalar product, 
when $a\in S^m(\R^{n-1},\R^{n-1}; E,F)$ 
and $E$, $F$ are in the scales $H^s(\R_+)$ or $H_0^s(\overline\R_+)$, $s\in \R$. 
It is easy to see that  $\op(a)^*=\op(a^{(*)})$  for an
$a^{(*)}\in S^m(\R^{n-1},\R^{n-1};F',E')$ with the asymptotic expansion
$$a^{(*)}(y,\eta) \sim\sum_{\alpha}\frac1{\alpha!}\partial^\alpha_\eta D_y^\alpha a(y,\eta)^*.$$ 
Here $a(y,\eta)^*\in \mathscr L(F',E')$ is the adjoint of $a(y,\eta)$. 

}
\def\cprime{$'$} \def\cprime{$'$} \def\cprime{$'$}


\begin{thebibliography}{10}

\bibitem{BCS14}
U.~Battisti, S.~Coriasco, and E.~Schrohe.
\newblock On a {C}lass of {F}ourier {I}ntegral {O}perators on {M}anifolds with
  {B}oundary.
\newblock arXiv:1406.0636.

\bibitem{BCS2}
U.~Battisti, S.~Coriasco, and E.~Schrohe.
\newblock On the {T}ransmission {P}roperty for {F}ourier {I}ntegral {O}perators
  of {B}outet de {M}onvel {T}ype.
\newblock In preparation.

\bibitem{BU71}
L.~Boutet~de Monvel.
\newblock Boundary problems for pseudo-differential operators.
\newblock {\em Acta Math.}, 126(1-2):11--51, 1971.

\bibitem{BMG81}
L.~Boutet~de Monvel and V.~Guillemin.
\newblock {\em The spectral theory of {T}oeplitz operators}, volume~99 of {\em
  Annals of Mathematics Studies}.
\newblock Princeton University Press, Princeton, NJ, 1981.

\bibitem{DAS01}
A.~Cannas~da Silva.
\newblock {\em Lectures on symplectic geometry}, volume 1764 of {\em Lecture
  Notes in Mathematics}.
\newblock Springer-Verlag, Berlin, 2001.

\bibitem{CH11}
H.~Christianson.
\newblock Classification of order preserving isomorphisms between algebras of
  semiclassical operators.
\newblock {\em Proc. Amer. Math. Soc.}, 139(2):499--510, 2011.

\bibitem{CO99b}
S.~Coriasco.
\newblock Fourier integral operators in SG classes II: Application to SG
  hyperbolic Cauchy problems.
\newblock {\em Ann. Univ. Ferrara, Sez. VII - Sc. Mat.}, 44:81--122, (1998),
  1999.

\bibitem{CO99}
S.~Coriasco.
\newblock Fourier integral operators in {SG} classes. {I}. {C}omposition
  theorems and action on {SG} {S}obolev spaces.
\newblock {\em Rend. Sem. Mat. Univ. Politec. Torino}, 57(4):249--302 (1999),
  2002.

\bibitem{DS76}
J.~J. Duistermaat and I.~M. Singer.
\newblock Order-preserving isomorphisms between algebras of pseudo-differential
  operators.
\newblock {\em Comm. Pure Appl. Math.}, 29(1):39--47, 1976.

\bibitem{EM98}
C.~Epstein and R.~Melrose.
\newblock Contact degree and the index of {F}ourier integral operators.
\newblock {\em Math. Res. Lett.}, 5(3):363--381, 1998.

\bibitem{EPI}
C.~L. Epstein.
\newblock Subelliptic {${\rm Spin}_{\mathbb C}$} {D}irac operators. {I}.
\newblock {\em Ann. of Math. (2)}, 166(1):183--214, 2007.

\bibitem{EPII}
C.~L. Epstein.
\newblock Subelliptic {${\rm Spin}_{\mathbb C}$} {D}irac operators. {II}.
  {B}asic estimates.
\newblock {\em Ann. of Math. (2)}, 166(3):723--777, 2007.

\bibitem{EPIII}
C.~L. Epstein.
\newblock Subelliptic {${\rm Spin}_{\mathbb C}$} {D}irac operators. {III}.
  {T}he {A}tiyah-{W}einstein conjecture.
\newblock {\em Ann. of Math. (2)}, 168(1):299--365, 2008.

\bibitem{GR87}
G.~Grubb.
\newblock Complex powers of pseudodifferential boundary value problems with the
  transmission property.
\newblock In {\em Pseudodifferential operators ({O}berwolfach, 1986)}, volume
  1256 of {\em Lecture Notes in Math.}, pages 169--191. Springer, Berlin, 1987.

\bibitem{GR96}
G.~Grubb.
\newblock {\em Functional Calculus for Boundary Value Problems}.
\newblock Birkah\"auser, Boston, 1996.
\newblock Reprint of the 1982 edition.

\bibitem{GH90}
G.~Grubb and L.~H{\"o}rmander.
\newblock The transmission property.
\newblock {\em Math. Scand.}, 67(2):273--289, 1990.

\bibitem{GU84}
V.~Guillemin.
\newblock Toeplitz operators in {$n$} dimensions.
\newblock {\em Integral Equations Operator Theory}, 7(2):145--205, 1984.

\bibitem{HP77}
A.~Hirschowitz and A.~Piriou.
\newblock La propri\'et\'e de transmission pour les distributions de {F}ourier;
  application aux lacunes.
\newblock In {\em S\'eminaire {G}oulaouic-{S}chwartz (1976/1977), \'{E}quations
  aux d\'eriv\'ees partielles et analyse fonctionnelle, {E}xp. {N}o. 14},
  page~19. Centre Math., \'Ecole Polytech., Palaiseau, 1977.

\bibitem{HO03}
L.~H{\"o}rmander.
\newblock {\em The analysis of linear partial differential operators {III}.
  Pseudo-differential operators}.
\newblock Classics in Mathematics. Springer, Berlin, 2007.
\newblock Reprint of the 1994 edition.

\bibitem{HO04}
L.~H{\"o}rmander.
\newblock {\em The analysis of linear partial differential operators {IV}.}
\newblock Classics in Mathematics. Springer, Berlin, 2007.
\newblock Reprint of the 1994 edition.

\bibitem{LSV94}
A.~Laptev, Y.~Safarov, and D.~Vassiliev.
\newblock On global representation of {L}agrangian distributions and solutions
  of hyperbolic equations.
\newblock {\em Comm. Pure Appl. Math.}, 47(11):1411--1456, 1994.

\bibitem{LNT01}
E.~Leichtnam, R.~Nest, and B.~Tsygan.
\newblock Local formula for the index of a {F}ourier integral operator.
\newblock {\em J. Differential Geom.}, 59(2):269--300, 2001.

\bibitem{MM12}
V.~Mathai and R.~B. Melrose.
\newblock Geometry of pseudodifferential algebras bundles and {F}ourier
  integral operators.
\newblock arXiv:1210.0990.

\bibitem{ME81}
R.~B. Melrose.
\newblock Transformation of boundary problems.
\newblock {\em Acta Math.}, 147(3-4):149--236, 1981.

\bibitem{NS06}
V.~E. Nazaikinskii and B.~Y. Sternin.
\newblock Surgery and the relative index in elliptic theory.
\newblock {\em Abstr. Appl. Anal.}, pages Art. ID 98081, 16, 2006.

\bibitem{NSS01}
V.~E. Naza{\u\i}kinski{\u\i}, B.~Y. Sternin, and B.-V. Shul{\cprime}tse.
\newblock The index of {F}ourier integral operators on manifolds with conical
  singularities.
\newblock {\em Izv. Ross. Akad. Nauk Ser. Mat.}, 65(2):127--154, 2001.

\bibitem{RS85}
S.~Rempel and B.-W. Schulze.
\newblock {\em Index theory of elliptic boundary problems}.
\newblock North Oxford Academic Publishing Co. Ltd., London, 1985.
\newblock Reprint of the 1982 edition.

\bibitem{S99}
E.~Schrohe.
\newblock Fr\'echet algebra techniques for boundary value problems on
  noncompact manifolds: {F}redholm criteria and functional calculus via
  spectral invariance.
\newblock {\em Math. Nachr.}, 199:145--185, 1999.

\bibitem{SC01}
E.~Schrohe.
\newblock A short introduction to {B}outet de {M}onvel's calculus.
\newblock In {\em Approaches to singular analysis ({B}erlin, 1999)}, volume 125
  of {\em Oper. Theory Adv. Appl.}, pages 85--116. Birkh\"auser, Basel, 2001.

\bibitem{SC91}
B.-W. Schulze.
\newblock {\em Pseudo-differential operators on manifolds with singularities},
  volume~24 of {\em Studies in Mathematics and its Applications}.
\newblock North-Holland Publishing Co., Amsterdam, 1991.

\bibitem{Seiler99}
J.~Seiler.
\newblock Continuity of edge and corner pseudodifferential operators.
\newblock {\em Math. Nachr.}, 205:163--182, 1999.

\bibitem{WE75}
A.~Weinstein.
\newblock Fourier integral operators, quantization, and the spectra of
  {R}iemannian manifolds.
\newblock In {\em G\'eom\'etrie symplectique et physique math\'ematique
  ({C}olloq. {I}nternat. {CNRS}, {N}o. 237, {A}ix-en-{P}rovence, 1974)}, pages
  289--298. \'Editions Centre Nat. Recherche Sci., Paris, 1975.
\newblock With questions by W. Klingenberg and K. Bleuler and replies by the
  author.

\bibitem{WE97}
A.~Weinstein.
\newblock Some questions about the index of quantized contact transformations.
\newblock {\em S\=urikaisekikenky\=usho K\=oky\=uroku}, (1014):1--14, 1997.
\newblock Geometric methods in asymptotic analysis (Japanese) (Kyoto, 1997).

\bibitem{ZE97}
S.~Zelditch.
\newblock Index and dynamics of quantized contact transformations.
\newblock {\em Ann. Inst. Fourier (Grenoble)}, 47(1):305--363, 1997.

\end{thebibliography}
\end{document}